\documentclass[11pt]{amsart}

\usepackage[
paper=a4paper,
text={147mm,230mm},centering
]{geometry}

\iftrue
\makeatletter
\def\@settitle{%
  \vspace*{-20pt}
  \begin{flushleft}%
    \baselineskip14\p@\relax
    \normalfont\bfseries\LARGE
    \@title
  \end{flushleft}%
}
\def\@setauthors{%
  \begingroup
  \def\thanks{\protect\thanks@warning}%
  \trivlist
  \large \@topsep30\p@\relax
  \advance\@topsep by -\baselineskip
  \item\relax
  \author@andify\authors
  \def\\{\protect\linebreak}%
  \authors
  \ifx\@empty\contribs
  \else
    ,\penalty-3 \space \@setcontribs
    \@closetoccontribs
  \fi
  \normalfont
  \endtrivlist
  \endgroup
}
\def\@setabstracta{%
    \ifvoid\abstractbox
  \else
    \skip@25\p@ \advance\skip@-\lastskip
    \advance\skip@-\baselineskip \vskip\skip@
    \box\abstractbox
    \prevdepth\z@ 
    \vskip-10pt
  \fi
}
\renewenvironment{abstract}{%
  \ifx\maketitle\relax
    \ClassWarning{\@classname}{Abstract should precede
      \protect\maketitle\space in AMS document classes; reported}%
  \fi
  \global\setbox\abstractbox=\vtop \bgroup
    \normalfont\small
    \list{}{\labelwidth\z@
      \leftmargin0pc \rightmargin\leftmargin
      \listparindent\normalparindent \itemindent\z@
      \parsep\z@ \@plus\p@
      
    }%
    \item[\hskip\labelsep\bfseries\abstractname.]%
}{%
  \endlist\egroup
  \ifx\@setabstract\relax \@setabstracta \fi
}
\def\section{\@startsection{section}{1}%
  \z@{-1.2\linespacing\@plus-.5\linespacing}{.8\linespacing}%
  {\normalfont\bfseries\large}}
\def\subsection{\@startsection{subsection}{2}%
  \z@{-.8\linespacing\@plus-.3\linespacing}{.3\linespacing\@plus.2\linespacing}%
  {\normalfont\bfseries}}
\def\subsubsection{\@startsection{subsubsection}{3}%
  \z@{.7\linespacing\@plus.1\linespacing}{-1.5ex}%
  {\normalfont\itshape}}
\def\@secnumfont{\bfseries}
\makeatother
\fi 

\usepackage{amssymb}
\usepackage{mathrsfs}
\usepackage[all]{xy}
\usepackage{graphicx,color,float}
\usepackage[bookmarks]{hyperref}
\usepackage{pinlabel}
\usepackage[textwidth=1.3in,color=yellow]{todonotes}
\usepackage{amsmath}
\usepackage{enumitem}
\usepackage{pb-diagram,pb-xy}
\usetikzlibrary{calc}
\usepackage{youngtab}
\usepackage[boxsize=.3 em]{ytableau}

\textwidth=\paperwidth \advance\textwidth by-2\oddsidemargin
\advance\oddsidemargin by-1in
\evensidemargin=\oddsidemargin
\calclayout

\theoremstyle{plain}
\newtheorem{theorem}{Theorem}[section]
 
\newtheorem*{thmx}{Theorem}

\newtheorem{proposition}[theorem]{Proposition}
\newtheorem{lemma}[theorem]{Lemma}
\newtheorem{corollary}[theorem]{Corollary}

\theoremstyle{definition}

\newtheorem{definition}[theorem]{Definition}
\newtheorem{example}[theorem]{Example}

\theoremstyle{remark}
\newtheorem{remark}[theorem]{Remark}

\newcommand{\C}{\mathbb{C}}

\newcommand{\Q}{\mathbb{Q}}

\newcommand{\R}{\mathbb{R}}

\newcommand{\Z}{\mathbb{Z}}

\newcommand{\CP}{\mathbb{C}P}

\def\pa{\partial}
\def\mcal{\mathcal}
\def\frak{\mathfrak}
\def\scr{\mathscr}

\numberwithin{equation}{section} \numberwithin{table}{section}

\def\to{\mathchoice{\longrightarrow}{\rightarrow}{\rightarrow}{\rightarrow}}
\makeatletter
\newcommand{\shortxra}[2][]{\ext@arrow 0359\rightarrowfill@{#1}{#2}}
\def\longrightarrowfill@{\arrowfill@\relbar\relbar\longrightarrow}
\newcommand{\longxra}[2][]{\ext@arrow 0359\longrightarrowfill@{#1}{#2}}

\makeatother
\numberwithin{equation}{section}

\begin{document}                                                                          
\title[Disk potential functions for quadrics]{Disk potential functions for quadrics}

\author{Yoosik Kim}
\address{Department of Mathematics, Pusan National University}
\email{yoosik@pusan.ac.kr}


\begin{abstract}
We compute the disk potential of Gelfand--Zeitlin monotone torus fiber in a quadric hypersurface by exploiting toric degenerations, Lie theoretical mirror symmetry, and the structural result of the monotone Fukaya category. 
\end{abstract}


\maketitle
\setcounter{tocdepth}{1} 
\tableofcontents

\section{Introduction}

The disk potential introduced by Fukaya--Oh--Ohta--Ono \cite{FOOObook1} is a generating function counting pseudo-holomorphic disks bounded by a Lagrangian submanifold. Since its introduction, the disk potential has been played a pivotal role in symplectic topology and mirror symmetry. When a Lagrangian submanifold $L$ is a torus and its disk potential is defined, the disk potential governs deformations of underlying Fukaya's $A_\infty$-algebra associated to $L$. In particular, it has been effectively employed to detect Lagrangian tori having non-trivial (deformed) Floer cohomology. In the context of SYZ mirror symmetry beyond Calabi--Yau manifolds, the disk potential provides a Landau--Ginzburg mirror of the ambient symplectic manifold according to Auroux \cite{AurouxTdual}. 

Computing disk potential functions is in general a non-trivial task as it requires classification of holomorphic disks and computation of counting invariants (especially when the abstract perturbation is turning on to achieve the transversality of moduli spaces). When it comes to the disk potentials, the most well-studied examples are toric manifolds and orbifolds. They are endowed with a Lagrangian torus fibration, consisting of free real torus orbits on a dense subset. Cho--Oh \cite{ChoOh} classified the holomorphic disks of Maslov index two bounded by each toric fiber. In particular, for compact Fano toric manifolds, the classification shows that the disk potential agrees with the Givental--Hori--Vafa superpotential, which can compute the quantum cohomology ring as its Jacobian ring. For general compact toric manifolds, Fukaya--Oh--Ohta--Ono \cite{FOOOToric1} constructed Lagrangian Floer theory on the de Rham model, see also \cite{WoodwardToric, BiranCornea, ChoPoddar}. Especially, they proved that the disk potential of a toric fiber can be defined and moreover the counting is invariant under equivariant perturbations so that the disk potential is well-defined (up to a choice of basis for the lattice). For toric Calabi--Yau manifolds and semi-Fano toric manifolds, their disk potential was obtained by computing non-trivial counting invariants from virtually perturbed moduli spaces, see \cite{ChanLauLeung, ChanLauLeungTseng} for instance.

A natural follow-up direction is to investigate algebraic varieties occurring at generic fibers of a toric degeneration since they can be understood via the central irreducible toric variety (\cite{AurouxSpecial, NishinouNoharaUeda, FOOOS2S2}). A particular type of toric degenerations such that the central toric variety admits a small toric resolution and each variety of the flat family is Fano was firstly studied. The toric degeneration of a partial flag variety (of type A) constructed by Gonciulea--Lakshmibai \cite{GonciuleaLakshmibai} is a prototype of such degenerations. Indeed, it was exploited to propose mirror construction of Calabi--Yau complete intersections in partial flag varieties in \cite{BCFKvS}. In this circumstance, Nishinou, Nohara, and Ueda constructed a Lagrangian torus fibration commuting with the toric moment map via a continuous extension of a densely defined symplectomorphism arising from the gradient Hamiltonian flow in \cite{WERuan2}. Moreover, they computed the disk potential of a Lagrangian torus. Indeed, the disk potential can be immediately written from the data of the facets of the lattice polytope associated to the central toric variety as in the case of Fano toric manifolds. 

However, there are still many interesting toric degenerations such that the central fiber does not admit any small toric resolutions. As a result, the disk potential of the Lagrangian fibration coming from a Newton--Okounkov body (and its corresponding toric degeneration) in \cite{HaradaKaveh} remains unknown. Moreover, it hinders Floer theoretical interpretation of superpotentials on the mirror Landau--Ginzburg model obtained by the mutations in the theory of cluster algebras and varieties.  

Some notable examples are in order. The Gelfand--Zeitlin system of every isotropic partial flag variety (of type B and D) having complex dimension $\geq 2$ in \cite{Thimm, GuilleminSternbergCCI} is an example. The system is compatible with the most ``standard" toric degeneration of isotropic partial flag variety. Namely, it converges to the toric moment map associated to the corresponding Gelfand--Zeitlin polytope. Secondly, by work of \cite{Littelmann, BerensteinZelevinsky, Caldero, HaradaKaveh}, each string polytope gives rise to a toric degeneration and a completely integrable system on a partial flag variety. Even on partial flag varieties of type A, there are completely integrable systems which do not satisfy the above condition corresponding to certain choices of reduced expression of the longest element of the Weyl group, see \cite[Remark 6.9]{ChoKimLeePark}. Another class of important examples includes bending systems on polygon spaces in \cite{KapovichMillson}. According to \cite{HausmannKnutson, NoharaUedaGrPoly}), the polygon space can be obtained from the symplectic reduction of the Grassmannian of two planes by the maximal torus. Moreover, the bending systems are descended from the (generalized) Gelfand--Zeitlin systems. Although each Gelfand--Zeitlin system \emph{does} fit into the above case, the corresponding bending system does \emph{not} in general. 

The main purpose of this paper is to enhance our understanding of disk potentials of generic fibers in toric degenerations beyond the setting in \cite{NishinouNoharaUeda}. The main difficulty lies on analyzing extra disk contribution intersecting lower dimensional strata. Indeed, the number of terms of disk potential is in general greater than that of facets of the polytope. Moreover, the degree of the evaluation map from the moduli space of stable maps with certain topological constraints to the Lagrangian, which is called a counting invariant or open Gromov--Witten invariant, can be {non-trivial}. Thus, computing the counting invariants is one of the essential steps.  

We briefly sketch the main idea how to compute the disk potential function of a smooth quadric hypersurface $\mcal{Q}_n$. The quadric $\mcal{Q}_n$ can be realized as an $\mathrm{SO}(n+2)$ coadjoint orbit. It has the Gelfand--Zeitlin (GZ) system, a completely integrable system constructed by Thimm's method \cite{Thimm, GuilleminSternbergCCI}. Indeed, this system degenerates into the toric moment associated to the Gelfand--Zeitlin polytope \cite{NishinouNoharaUeda2}. Using the toric degeneration of completely integrable systems, we classify a possible list of homotopy classes which can contribute to the disk potential of the torus fiber of the GZ system. As the central fiber is a toric orbifold, the classification result in \cite{ChoPoddar} of holomorphic (orbi-)disks helps us to compute counting invariants of most homotopy classes. We then apply a structural result of the monotone Fukaya category \cite{SheridanFano} and Lie theoretical mirror symmetry \cite{PechRietschWilliams} to compute the remaining non-trivial counting invariants.

\begin{thmx}[Theorem~\ref{theorem_main} and Corollary~\ref{corollary_nondisplaceable}]
The disk potential of a monotone Lagrangian Gelfand--Zeitlin torus fiber in a quadric hypersurface $\mcal{Q}_n$ of $\CP^{n+1}$ is 
$$
W(\mathbf{z}) = \frac{1}{z_n} + \frac{z_n}{z_{n-1}} + \dots + \frac{z_3}{z_2} + \frac{z_{2}}{z_{1}} + 2 z_2 + z_1 z_2.
$$

In particular, the monotone Lagrangian Gelfand--Zeitlin torus fiber can not be displaceable by any Hamiltonian diffeomorphism.
\end{thmx}

\begin{remark}
\begin{itemize}
\item The disk potential agrees with the Przhiyalkovski\u{\i}'s Landau--Ginzburg mirror \cite{Przhiya}, which is a cluster chart of Pech--Rietsch--Williams's Landau--Ginzburg mirror in \cite{PechRietschWilliams}, see Corollary~\ref{corollary_Pzrmirror}.
\item When $n = 2$, the disk potential of $\mcal{Q}_2 \simeq \CP^1 \times \CP^1$ was computed by Auroux  and Fukaya--Oh--Ohta--Ono \cite{AurouxSpecial, FOOOS2S2}. 
\item Homological mirror symmetry of quadric hypersurfaces (more generally Fano hypersurfaces) was proven by Sheridan \cite{SheridanFano}. He employed an immersed Lagrangian sphere with bounding cochains to split-generate the monotone Fukaya category and the non-zero objects were classified into two groups$\colon$ objects corresponding to a small eigenvalue and objects corresponding to a big eigenvalue. The monotone torus $L$ in Theorem~\ref{theorem_main} together with $\nabla$ corresponding to a critical point is a non-zero object with small eigenvalue. Moreover, by the structural result in \cite{SheridanFano}, each non-zero object $(L, \nabla)$ split-generates one summand of the quantum cohomology with a non-zero eigenvalue. 
\end{itemize}
\end{remark}

\subsection*{Acknowlodgement}

The author express his deep gratitude to Yunhyung Cho, Hansol Hong and Siu-Cheong Lau.
The paper grows out from the collaborations and discussions with them. The author would like to thank the anonymous referees for the helpful comments. This work was supported by the National Research Foundation of Korea (NRF) grant funded by the Korea government NRF-2021R1F1A1057739 and NRF-2020R1A5A1016126, and Pusan National University Research Grant, 2021.

\section{Preliminaries on quadrics}

Let $\mcal{Q}_n$ be a smooth quadric hypersurface in $\CP^{n+1}$ for $n \geq 2$. 
In this section, we collect some results on $\mcal{Q}_n$ which will be used in this manuscript later on.

\subsection{Topology of quadrics}

Any two smooth hypersurfaces of the same degree equipped with the symplectic form from the ambient projective space are symplectomorphic. In particular, every two dimensional smooth quadric hypersurface $\mcal{Q}_2$ is diffeomorphic to $\CP^1 \times \CP^1$ via the Segre embedding. By applying the Lefschetz hyperplane theorem to the Veronese embedding of $\CP^{n+1}$ with $n \geq 3$, we obtain the following lemma.

\begin{lemma}\label{lemma_pi2}
For $n \geq 3$, the $2$nd homotopy group of $\mcal{Q}_n$ is isomorphic to $\Z$, i.e. $\pi_2(\mcal{Q}_n) \simeq \Z$.
\end{lemma}

Every quadric hypersurface $\mcal{Q}_n$  is defined by a quadratic form associated to a symmetric bilinear form $\mcal{B}$ on the complex vector space $\C^{n+2}$. That is, $\mcal{Q}_n$ is the vanishing locus of $\mcal{B} (\mathbf{x}, \mathbf{x})$ in $\CP^{n+1}$ where $\mathbf{x} := (x_0, \dots, x_{n+1})$. Thus, the quadric hypersurface $\mcal{Q}_n$ can be viewed as a complex orthogonal Grassmannian $\mathrm{OG}(1, \C^{n+2})$, the space of one dimensional isotropic subspace of $\C^{n+2}$ when $n \geq 1$. Also, $\mcal{Q}_n$ is isomorphic to the quotient $G_{\C} / P$ by a parabolic subgroup $P$ where $G_\C$ is the complex special orthogonal group $\mathrm{SO}(n+2; \C)$.

The cohomology group of the orthogonal Grassmannian $\mathrm{OG}(1, \C^{n+2})$ (more generally partial flag varieties) has a distinctive additive basis consisting of Schubert classes. From the combinatorics of the parabolic Weyl group of the complex Lie group, we count the number of Schubert classes of $\mathrm{OG}(1, \C^{n+2})$, which leads to the following topological fact.

\begin{lemma} The sum of Betti numbers of $\mcal{Q}_n$ is 
$$
\begin{cases}
n+1 &\mbox{if $n$ is odd},\\ 
n+2 &\mbox{if $n$ is even.} 
\end{cases}
$$
\end{lemma}

Consider the coadjoint action of the real special orthogonal group $G = \mathrm{SO}(n+2;\R)$ on the dual Lie algebra $\frak{g}^* \simeq \frak{so}(n+2)^*$. Identifying $\frak{g}^* \simeq \frak{g}$ via the Killing form 
\begin{equation}\label{equ_killingforms}
\langle X, Y \rangle := - \frac{1}{2} \mathrm{Tr}(XY)
\end{equation}
on $\frak{g}$, the coadjoint action on $\frak{g}^*$ can be regarded as the adjoint action on the set $\frak{g}$ of $(n+2) \times (n+2)$ skew-symmetric matrices. 

Since $G_{\C} / P \simeq G / K$ where $G$ is the compact real form of $G_{\C} = \mathrm{SO}(n+2;\C)$ and $K = P \cap G$, we have $G/K \simeq \mcal{Q}_{n}$. To realize $\mcal{Q}_{n}$ as a coadjoint orbit, we choose a $\nu$-tuple of real numbers of the following form$\colon$
\begin{equation}\label{equ_givenlambda}
\lambda = (\lambda_1 > \lambda_2 := 0, \dots, \lambda_\nu := 0)
\end{equation}
where $\nu := \lfloor (n+2)/2 \rfloor$. Let $B(\lambda_i)$ be the $(2 \times 2)$ skew-symmetric matrix whose $(1,2)$ entry is $\lambda_i$, that is,
\begin{equation*}
B(\lambda_i) = 
\begin{bmatrix}
0 & \lambda_i \\
-\lambda_i & 0 
\end{bmatrix}.
\end{equation*}
Via the identification given by~\eqref{equ_killingforms}, the coadjoint orbit $\mcal{O}_\lambda^n$ is then defined by the orbit of the $(n+2) \times (n+2)$ block diagonal matrix $D^n_\lambda$ under the adjoint $G$-action where
$$
D^n_\lambda = \begin{cases}
\mathrm{diag}(B(\lambda_1), \dots, B(\lambda_\nu)) &\quad \mbox{if $n$ is even,} \\
\mathrm{diag}(B(\lambda_1), \dots, B(\lambda_\nu), 0_{1 \times 1}) &\quad \mbox{if $n$ is odd.}
\end{cases}
$$

When $n \geq 1$, the orbit $\mcal{O}_\lambda^n$ is the space of skew-symmetric matrices whose eigenvalues are $\pm \lambda_1 \sqrt{-1}, 0, \dots, 0$. The coadjoint orbit becomes a symplectic manifold equipped with the Kirillov--Kostant--Souriau (KKS) form $\omega^n_\lambda$. 

\subsection{Completely integrable systems on quadrics}\label{subsection_cominsysonquad}

To construct a Lagrangian torus fibration, we shall employ a Gelfand--Zeitlin (GZ) system, a completely integrable system on $\mcal{O}_\lambda^n \simeq \mcal{Q}_n$. Its construction is recalled below.

For an $(n+2) \times (n+2)$ square matrix $A$ and an integer $j$ satisfying $2 \leq j \leq n+2$, we denote by $A^{(j)}$ the leading $(j \times j)$ principal submatrix. If $A \in \mcal{O}_\lambda^n \simeq \mcal{Q}_n$, then the submatrix $A^{(j)}$ is also skew-symmetric so that the eigenvalues of $A^{(j)}$ are either all zero or $\pm \lambda^{(j)}_1 \sqrt{-1}, 0, \dots, 0$ for a positive real number $\lambda^{(j)}_1:= \lambda^{(j)}_1(A)$ by the min-max principle. Here, the min-max principle refers to a system of inequalities between eigenvalues of submatrices of a skew-symmetric matrix, see \cite[Chapter 5]{Pabiniak} for instance. When the eigenvalues of $A^{(j)}$ are all zero, we set $\lambda^{(j)}_1(A) = 0$. We then define 
$$
\Phi^n_\lambda := (\Phi_1, \Phi_2, \dots, \Phi_{n}) \colon \mcal{O}_\lambda^n \to \R^n
$$ 
by
\begin{equation}\label{equ_gssys}
\Phi_{j}(A) := 
\begin{cases}
\lambda^{(j+1)}_1(A) &\mbox{if either $j \geq 2$ or ($j = 1$ and $\mathrm{Pf}(A^{(2)}) \geq 0$)} \\
- \lambda^{(2)}_1(A) &\mbox{if $j = 1$ and $\mathrm{Pf}(A^{(2)}) < 0$}
\end{cases}
\end{equation}
where $\mathrm{Pf}(A)$ is the \emph{Pfaffian} of a skew-symmetric matrix $A$. Notice that the indices for $\Phi_\bullet$ and $\lambda^{(\bullet)}_1$ differ by one in order to make a coordinate system for the image of $\Phi_\bullet$ start from one.

Under the identification~\eqref{equ_killingforms}, the system $\Phi^n_\lambda$ agrees with the completely integrable system consisting of collective Hamiltonians constructed by Guillemin and Sternberg \cite{GuilleminSternbergCCI}. It is obtained by applying the Thimm's method \cite{Thimm} to the sequence 
$$
G = \mathrm{SO}(n+2) \supset \mathrm{diag}(\mathrm{SO}(n+1), I_{1}) \supset \mathrm{diag}(\mathrm{SO}(n), I_{2}) \supset \dots \supset \mathrm{diag}(\mathrm{SO}(2), I_{n})
$$
where $I_k$ is the $(k \times k)$ identity matrix. In general, the component $\Phi_j$ is \emph{not} globally smooth. Nevertheless, it is smooth on the open dense set $(\Phi^n_\lambda)^{-1}(\{\mathbf{u} \in \R^{n} \colon u_j \neq 0\})$.  Moreover, according to \cite{GuilleminSternbergGC}, each component of the GZ system is a moment map of a Hamiltonian $S^1$-action on its smooth locus.

\begin{theorem}[\cite{GuilleminSternbergGC, GuilleminSternbergCCI}]\label{theorem_GS}
The map $\Phi^n_\lambda$ in~\eqref{equ_gssys} is a completely integrable system on $\mcal{O}_\lambda^n$. The component $\Phi_j$ generates a Hamiltonian $S^1$-action on the open subset $(\Phi^n_\lambda)^{-1}(\{\mathbf{u} \in \R^{n} \colon u_j \neq 0\})$.
\end{theorem}
 
The image of a coadjoint orbit under the map $\Phi^n_\lambda$ is a polytope, which is determined by the min-max principle. The image is called a Gelfand--Zeitlin(GZ) polytope and denoted by $\Delta^n_\lambda$. In this particular case where $\mcal{O}_\lambda^n \simeq \mcal{Q}_n$, the image is a simplex determined by
\begin{equation}\label{equ_GZpolytope}
\lambda_1 \geq u_n \geq u_{n-1} \geq \dots \geq u_2 \geq |u_1|
\end{equation}
where $(u_1, u_2, \dots, u_n)$ is the coordinate system for $\R^n$. We can always adjust the size of the polytope by scaling the KKS form so that $\lambda_1$ is assumed to be $n$ without any loss of generality.

Indeed, the GZ system has a peculiar feature$\colon$ every GZ fiber (even over a critical value) is a smooth isotropic submanifold. The topology of GZ fibers was described in terms of iterated bundles, see \cite{ChoKimSO, Lanesympcont}. In the case where $\mcal{O}_\lambda^n \simeq \mcal{Q}_n$, according to the combinatorial test devised in \cite{ChoKimSO}, the non-torus GZ fibers can occur only over the stratum of codimension two determined by $u_2 = 0$, which will be important to us. The GZ fiber over each point in the relative interior of that stratum is a Lagrangian submanifold diffeomorphic to ${S}^2 \times {T}^{n-2}$. 

\begin{example}
When $n = 3$ and $\lambda = (3, 0, 0)$, the coadjoint orbit $\mcal{O}_\lambda^3$ has three types of non-torus fibers. 
\begin{itemize}
\item (The blue point, $u_1 = u_2 = u_3 = 0$) The GZ fiber is ${S}^3$.
\item (The purple point, $u_1 = u_2 = 0, u_3 = 3$) The GZ fiber is ${S}^2$.
\item (Any point in the red strata $u_1 = u_2 = 0 < u_3 < 3$) The GZ fiber is ${S}^2 \times {S}^1$.
\end{itemize}

\begin{figure}[h]
	\scalebox{0.9}{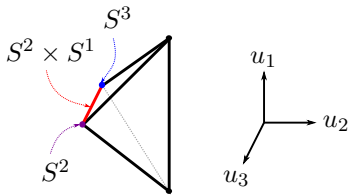}
	\caption{\label{fig_gzfiber} Non-torus fibers of the Gelfand--Zeitlin system in $\mcal{O}_\lambda^3$}	
\end{figure}
\end{example}

\begin{remark}
Precisely speaking, to show that the GZ fiber over a point in the relative interior of $u_2 = 0$ is a product of $S^2$ and $ {T}^{n-2}$, one needs to use the argument in \cite{ChoKimOh1}.
\end{remark}

\subsection{Toric degenerations of quadrics}\label{sec_toricdeg}

The toric degenerations of completely integrable systems were constructed by Nishinou--Nohara--Ueda \cite{NishinouNoharaUeda} and Harada--Kaveh \cite{HaradaKaveh}, using a technique of the gradient Hamiltonian flow in \cite{WERuan2}. In general, there are multiple toric degenerations of partial flag varieties (and completely integrable systems compatible with toric degenerations). In this subsection, we shall specify a toric degeneration which is compatible with the Gelfand--Zeitlin system and polytope for $\mcal{Q}_n$, which was studied in \cite{NishinouNoharaUeda2}. 

Take the following $(n+2) \times (n+2)$ block diagonal matrix 
\begin{equation*}
\mathrm{diag}\left(
\begin{bmatrix}
0 & 1 \\
1 & 0 
\end{bmatrix}, I_{n}
\right)
\end{equation*}
as a bilinear form $\mcal{B}$. The bilinear form $\mcal{B}$ gives rise to a smooth quadric hypersurface$\colon$
$$
\mcal{Q}_n = \left\{ [x_0 \colon x_1 \colon \dots \colon x_{n+1} ] \in \CP^{n+1} ~\colon~ 2x_0x_1 + x_2^2 + \dots + x_{n+1}^2 = 0 \right\}.
$$
To match up with the KKS form coming from the choice~\eqref{equ_givenlambda}, $\mcal{Q}_n$ is adorned with a multiple of the Fubini--Study form $\omega_{FS}$ on $\CP^{n+1}$. 

The following family $\{ \mcal{X}_\mathbf{t} \}_{\mathbf{t} \in \C^{n-1}}$ where
\begin{equation}\label{equ_familyofquadrics}
\mcal{X}_{\mathbf{t}} :=  \left\{ [x_0 \colon x_1 \colon \dots \colon x_{n+1} ] \in \CP^{n+1} ~\colon~ 2x_0x_1 + x_2^2 + t_3 x_3^2 + t_4 x_4^2 + \dots + t_{n+1} x_{n+1}^2 = 0 \right\}
\end{equation}
is a flat family of irreducible algebraic varieties in $\CP^{n+1}$. Notice that the central fiber $\mcal{X}_{\mathbf{0}}$ is a toric variety. Letting $\| \mathbf{x} \| := |x_0|^2 + \dots + |x_{n+1}|^2$,  a toric moment map can be chosen as
\begin{equation}\label{equ_toricmoment}
\Phi_{\mathrm{toric}} (\mathbf{x}) = \frac{\lambda_1}{\| \mathbf{x} \|} \left( {-|x_0|^2 + |x_1|^2}, \sum_{j=0}^2 |x_j|^2, \sum_{j=0}^3 |x_j|^2, \dots, \sum_{j=0}^{n} |x_j|^2 \right),
\end{equation}
so that the image is exactly the GZ polytope $\Delta^n_\lambda$ determined by~\eqref{equ_GZpolytope}.

Furthermore, the toric degeneration $\{ \mcal{X}_{\mathbf{t}} \}$ is also compatible with the GZ system in the following sense. 

\begin{proposition}[Proposition 3.1 in \cite{NishinouNoharaUeda2}]\label{proposition_gradientflow}
Let $\mcal{X}_{\mathbf{t}}$ be the toric degeneration given in~\eqref{equ_familyofquadrics}.
We then have the following$\colon$
\begin{enumerate}
\item A piecewise smooth path $\gamma \colon [0,1] \to \C^{n-1}$ such that $\gamma(1) = (1,1,\dots, 1)$ and $\gamma(0) = \mathbf{0}$.
\item For each pair $(s, t) \in [0,1]^2$ with $s \leq t$, a surjective continuous map $\Psi_{t,s} \colon \mcal{X}_{(t)} \to \mcal{X}_{(s)}$ where $\mcal{X}_{(t)} := \mcal{X}_{\gamma(t)}$, 
\item For each $t \in [0,1]$, a completely integrable system $\Phi_{(t)} \colon \mcal{X}_{(t)} \to \Delta^n_\lambda$ satisfying
\end{enumerate}
\begin{itemize}
\item The map $\Phi_{(0)}$ agrees with the toric moment map in~\eqref{equ_toricmoment},
\item The map $\Phi_{(1)}$ agrees with the Gelfand--Zeitlin system in~\eqref{equ_gssys},
\item Setting $\mathring{\mcal{X}}_0 := \mcal{X}_0 \backslash \mathrm{Sing}(\mcal{X}_0)$ and $\mathring{\mcal{X}}_t := \Psi_{t,0}^{-1}(\mathring{\mcal{X}}_0)$, the map $\Psi_{t,0}$ is a symplectomorphism on $\mathring{\mcal{X}}_t$,
\item The map $\Psi_{t,s}$ makes the following diagram commute.
$$
\xymatrix{
\mcal{X}_{(t)} \ar[dr]_{\Phi_{(t)}} \ar[rr]^{\Psi_{t,s}}  & & \mcal{X}_{(s)} \ar[dl]^{\Phi_{(s)}} \\
& \Delta_\lambda^n & 
}
$$
\end{itemize}
\end{proposition}

\subsection{Lie theoretical Landau--Ginzburg mirrors of quadrics}\label{mirrorsymmofquadricss}

Rietsch \cite{Rietsch} constructed the Lie theoretical mirror of partial flag varieties. When $X$ is the quadric hypersurface $\mcal{Q}_n$, Pech--Rietsch--Williams in \cite{PechRietschWilliams} found a malleable expression and cluster structure of the Lie theoretical mirror, which will be recalled below.

For the odd dimensional quadric $X = \mcal{Q}_{2m-1}$, consider $\mathbb{P}(H^*(Q_{2m-1}, \C)^*) \simeq \CP^{2m-1}$ where $p_\bullet$'s are the homogeneous coordinates for $\mathbb{P}(H^*(Q_{2m-1}, \C)^*)$ corresponding to the Schubert basis. Then the Landau--Ginzburg mirror of the odd dimensional quadric $X$ consists of
\begin{enumerate}
\item The mirror complex manifold $\check{X}$ is the complement of $\scr{D}_0 + \dots + \scr{D}_m$ in $\CP^{2m-1}$ where
\begin{itemize}
\item $\scr{D}_0 := \{p_0 = 0\},$
\item $\scr{D}_j := \left\{ \delta_j:= \sum_{i=0}^j (-1)^i p_{j-i} \, p_{2m-1-j+i} = 0\right\} \mbox{for $j = 1, \dots, m-1$},$
\item $\scr{D}_m = \{p_{2m-1} =0\}$.
\end{itemize}
\item The superpotential $W_q \colon \check{X} \to \C$ is given by
\end{enumerate}
$$
W_q (\mathbf{p}) =  \frac{p_1}{p_0} + \sum_{j = 1}^{m-1} \frac{p_{j + 1} \, p_{2m-1 - j}}{\delta_j} + q \frac{p_1}{p_{2m-1}}.
$$

For the even dimensional quadric $X = \mcal{Q}_{2m-2}$, consider the dual quadric $\check{\mcal{Q}}_{2m-2}$, defined by
$$
p_{m-1}p_{m-1}^\prime - p_m p_{m-1} + \cdots + (-1)^{m-1} p_{m-2} p_0 = 0
$$
where $p_\bullet$'s and $p_{m-1}^\prime$ are the homogeneous coordinates for $\mathbb{P}(H^*(Q_{2m-2}, \C)^*)$ corresponding to the Schubert basis. Then the Landau--Ginzburg mirror consists of 
\begin{enumerate}
\item The mirror complex manifold $\check{X}$ is the complement of $\scr{D}_0 + \scr{D}_1 + \dots + \scr{D}_{m-2} + \scr{D}_{m-1} + \scr{D}_{m-1}^\prime$ in the dual quadric $\check{\mcal{Q}}_{2m-2}$ where
\begin{itemize}
\item $\scr{D}_0 := \{p_0 = 0\},$
\item $\scr{D}_j := \left\{ \delta_j:= \sum_{i=0}^j (-1)^i p_{j-i} \cdot p_{2m-2-j+i} = 0\right\} \mbox{for $j = 1, \dots, m-3$},$
\item $\scr{D}_{m-2} = \{p_{2m-2} =0\}$,
\item $\scr{D}_{m-1} = \{p_{m-1} =0\}$,
\item $\scr{D}_{m-1}^\prime = \{p^\prime_{m-1} =0\}$.
\end{itemize}
\item The superpotential $W_q \colon \check{X} \to \C$ is given by
$$
W_q (\mathbf{p}) =  \frac{p_1}{p_0} + \sum_{j = 1}^{m-3} \frac{p_{j + 1} \, p_{2m- 2 - j}}{\delta_j} + \frac{p_m}{p_{m-1}} + \frac{p_m}{p^\prime_{m-1}} + q \frac{p_1}{p_{2m-2}}.
$$
\end{enumerate}

\begin{theorem}[Theorem 2.1 and 3.2 in \cite{PechRietschWilliams}]
The above pair $(\check{X}, W_q)$ is the Landau--Ginzburg mirror of the quadric $\mcal{Q}_n$. In particular, the Jacobian ring of $(\check{X}, W_q)$ is isomorphic to the quantum cohomology ring of $\mcal{Q}_n$ $($with the quantum parameter
inverted$)$.
\end{theorem}

We then can compute the critical values of the superpotentials, which are mirror to the eigenvalues of the quantum cup product with $c_1(TX) $ on the quantum cohomology. The following result will be used for computing counting invariants later on.

\begin{proposition}[Proposition 2.3 and 3.13 in \cite{PechRietschWilliams}]\label{cor_criticalvalues}
The critical values of the superpotential $W_q$ mirror to $\mcal{Q}_n$ are $\left\{ n \cdot \zeta \colon \zeta^n = 4q \right\} \cup \{0\}$. 
\end{proposition}

\section{Lagrangian Floer theory and monotone Fukaya category}

The goal of this section is to recall the definition of disk potential functions and the structural result on the monotone Fukaya category. Also, using the \emph{gradient holomorphic disks}, we shall prove that the Gelfand--Zeitlin torus fiber at the ``center" of the Gelfand--Zeitlin polytope $\Delta^n_\lambda$ is monotone.

\subsection{Monotone Lagrangian submanifolds}

A symplectic manifold $(X, \omega)$ is called \emph{monotone} if $[\omega] = \eta \cdot c_1 (TX)$ for some positive number $\eta$ in $H^2(X; \R)$. A Lagrangian submanifold $L$ of $X$ has two group homomorphisms$\colon$
\begin{itemize}
\item $I_\omega \colon \pi_2(X,L) \to \R$ is given by the symplectic area,
\item $I_\mu \colon \pi_2(X,L) \to \Z$ is given by the Maslov index.
\end{itemize}
A Lagrangian submanifold $L$ is said to be \emph{monotone} if there exists a positive number $\upsilon$ such that $I_\omega (\beta) = \upsilon \cdot I_\mu (\beta)$ for every $\beta \in \pi_2(X,L)$.
A monotone Lagrangian submanifold can exist only in a monotone symplectic manifold. Moreover, two monotonicity constants $\eta$ and $\upsilon$ are related by $\eta = 2 \upsilon$ unless $X$ is symplectically aspherical, that is, $c_1(TX)([\alpha]) = 0$ for every $\alpha \in \pi_2(X)$.

Oh in \cite{OhMonotone, OhMonotone2} introduced the notion of monotonicity of Lagrangian submanifolds and constructed Floer cohomology of monotone Lagrangians. Under the monotonicity assumption, Floer homology on the pearl complex of $L$ which was firstly employed in \cite{OhPearl} was constructed by Biran--Cornea \cite{BiranCornea} in details.

\begin{remark}
Lagrangian Floer theory and their deformations were constructed beyond the monotone setting by Fukaya--Oh--Ohta--Ono \cite{FOOObook1, FOOOMorse, FukayaCyclic, FOOOToric1} for simplicial models, Morse models, and de Rham models, see also Charest--Woodward \cite{CharestWoodward} for rational symplectic manifolds. \end{remark}

Let $X$ be a monotone closed symplectic manifold. For a choice of compatible almost complex structure $J$ on $X$, we denote by $\mcal{M}_1(X, L, \beta; J)$ the moduli space of stable $J$-holomorphic maps from a bordered genus $0$ Riemann surface in $\beta \in \pi_2(X,L)$ with one boundary marking point. 

Suppose that a Lagrangian submanifold $L$ is monotone, orientable, and (relative) spin. Also, assume that the monotonicity constant $\upsilon$ is positive. To orient the moduli spaces and determine signs for the underlying $A_\infty$ structure maps, we make a choice for orientation and (relative) spin structure of $L$. A generic compatible almost complex structure $J$ makes every disk class of Maslov index two Fredholm regular.  Moreover, the monotonicity condition ensures that it does not bound any non-constant holomorphic disks of Maslov index zero. Consequently, the moduli space $\mcal{M}_1(X, L, \beta; J)$ for each class $\beta$ of Maslov index two becomes a manifold of dimension $\dim_\R L$ without boundary. The counting $n_\beta(L, J)$ is defined by the degree of the evaluation map at the marking point $z_0$. 
\begin{equation}\label{equ_evaluationev0}
\mathrm{ev}_{0} \colon \mcal{M}_1(X, L, \beta; J) \to L \quad \mbox{$\varphi \mapsto \varphi(z_0)$}.
\end{equation}
This number $n_\beta(L, J)$ is called the \emph{open Gromov--Witten invariant} of $\beta$.

In this circumstance, this counting is \emph{invariant} under a choice of generic compatible almost complex structures. A generic path from $J$ to $J^\prime$ in the space of compatible almost complex structures produces an oriented cobordism between $\mcal{M}_1(X, L,\beta;J)$ and $\mcal{M}_1(X, L,\beta;J^\prime)$. The monotonicity condition prohibits bubbles a priori through the cobordism. This argument gives the following lemma, see \cite{EliashbergPolterovich} for instance. 

\begin{lemma}\label{lemma_invarianceofcounting}
Assume that $L$ is a monotone Lagrangian submanifold. For two generic compatible almost complex structures $J$ and $J^\prime$, the counting numbers $n_\beta(L,J)$ and $n_\beta(L,J^\prime)$ are equal. 
\end{lemma}

For the sake of notational simplicity, we set $\mcal{M}_1(\beta) := \mcal{M}_1(X, L,\beta;J)$ and $n_\beta = n_\beta(L,J)$ (with generic $J$).

\subsection{Disk potentials and the monotone Fukaya category}

In general, a Lagrangian submanifold $L$ may bound a holomorphic disk, which causes an obstruction on defining Floer cohomology. Assume that $L$ is a monotone Lagrangian torus of a compact non-symplectically aspherical manifold $X$. In this case, the obstruction does not cause any trouble for constructing Floer cohomology although it may bound a holomorphic disk. As we have just observed, the moduli space $\mcal{M}_1(\beta)$ represents a cycle and the push-forward of the evaluation map in~\eqref{equ_evaluationev0} for any class $\beta$ is proportional to the fundamental cycle of $L$. It in turn implies that Floer cohomology is well-defined (i.e. the composition of consecutive differentials vanishes) as the fundamental cycle of $L$ is a unit of the underlying $A_\infty$-algebra associated to $L$. 

 To incorporate deformations of Floer theory, we adorn the trivial line bundle $\scr{L}$ over a Lagrangian torus $L$ and take a flat non-unitary $\C^*$-connection $\nabla$ on the bundle as in \cite{Chononunitrary}. 
Choose a set of oriented loops $\{ \vartheta_1, \dots, \vartheta_n \}$ that forms a basis for $\pi_1(L) \simeq \Z^n$ as a $\Z$-module. Two connections $\nabla$ and $\nabla^\prime$ are called \emph{gauge equivalent} if $\mathrm{hol}_\nabla (\vartheta_j) =  \mathrm{hol}_{\nabla^\prime} (\vartheta_j)$ for every $j = 1, \dots, n$. We denote by $\mcal{M}(\scr{L})$ the moduli space of flat $\C^*$-connections modulo gauge equivalence. 
The above choice of the oriented loops leads to a coordinate system on $\mcal{M}(\scr{L})$ consisting of the holonomy variables
\begin{equation}\label{equ_zjvariable}
z_j := \mathrm{hol}_\nabla (\vartheta_j) \in \C^* \quad \mbox{ ($j = 1, \dots, n$)}.
\end{equation}
In other words, it determines an identification between $\mcal{M}(\scr{L})$ and an algebraic torus $(\C^*)^n$. From the point of view of SYZ mirror symmetry, the moduli space $\mcal{M}(\scr{L})$ of $\C^*$-flat connections serves as a mirror complex chart, see \cite{AurouxTdual}.

The obstruction of Lagrangian Floer theory is expressed as
\begin{equation}\label{equ_obstrLFT}
\frak{m}^\nabla_0(1) = \sum_{\beta} \mathrm{hol}_\nabla (\partial \beta) \cdot (\mathrm{ev}_0)_* [\mcal{M}_1(\beta)] \cdot T^{\omega(\beta)}
= \sum_{\beta} n_\beta \cdot \mathrm{hol}_\nabla (\partial \beta) \cdot [L] \cdot T^{\omega(\beta)}.
\end{equation}
\begin{remark}
In~\eqref{equ_obstrLFT}, the variable $T$ is the formal parameter recording symplectic area of $J$-holomorphic curves.
More precisely, we need to work over the Novikov ring given by
$$
\Lambda_0 =\left\{\sum_{i=0}^\infty a_i T^{A_i} \mid A_i \geq 0 \textrm{ increase to } +\infty, a_i \in \C \right\}
$$ 
in order to define Lagrangian Floer theory for general Lagrangian submanifolds. 
As dealing with a monotone Lagrangian submanifold in a closed symplectic manifold, we may work over the field of complex numbers by replacing $T$ with $e^{-1}$ for instance. That is, there is no convergence issues.
\end{remark}

Notice the Maslov index of $\beta$ should be two for $n_\beta$ being non-zero. Since $L$ is monotone, $\omega(\beta)$ is equal to $2 \upsilon$ for every $\beta$ with $\mu(\beta) = 2$. It means that the valuation of each term in $\frak{m}^\nabla_0(1)$ is identical. Notice that $\mathrm{hol}_\nabla (\partial \beta)$ can be expressed as a Laurent monomial in terms of the variables~\eqref{equ_zjvariable}.

\begin{definition}\label{def_diskpotentfunct}
The \emph{disk potential function} of $L$ is defined by
$$
W_L \colon \mcal{M}(\scr{L}) \simeq (\C^*)^n \to \C \quad [\nabla] \simeq (z_1, \dots, z_n) \mapsto \sum_{\beta; \mu(\beta) = 2}  n_\beta \cdot z^{\partial \beta}.
$$
\end{definition}

The derivative of $W_L$ determines the Floer differentials of the $1$-cochains $\{  \vartheta^\vee_1, \dots, \vartheta^\vee_n \}$. In particular, the Floer boundary map deformed by $\nabla$ corresponding to a critical point vanishes on the $1$-cochains. Furthermore, the $A_\infty$-relation yields that the differential of every cochain is $0$ so that the Floer cohomology is isomorphic to the cohomology of $T^n$. Such an object $(L, \nabla)$ becomes a \emph{non-zero} object in the monotone Fukaya category. More precisely, $(L, \nabla)$ is meant to be the Lagrangian torus $L$ together with the class represented by a flat $\C^*$-connection $\nabla$ and a choice of spin structure and orientation. For more details, the reader is referred to \cite{FOOOToric1, AurouxTdual} for instance.

We now recall the structural result on the monotone Fukaya category proven by Sheridan \cite{SheridanFano}. The result will be employed to compute some counting invariants later on. The values of the disk potential of $(L, \nabla)$ provide a decomposition on the Fukaya category. Namely, the monotone Fukaya category is a direct summation of $\mcal{F}uk(X)_w$ where $\mcal{F}uk(X)_w$ is the full subcategory whose objects are $(L, \nabla)$ with the value $W_L (\nabla) = w$.

Regarding the quantum cohomology $QH(X, \omega)$ as the vector space over $\C$, the first Chern class $c_1 (TX)$ acts on the vector space $QH(X, \omega)$ via the quantum cup product. As the Fukaya category communicates with the closed string mirror symmetry (via closed-open maps or open-closed maps), the critical values making $\mcal{F}uk(X)_w$ non-trivial are determined by the eigenvalues of the quantum cohomology $QH(X, \omega)$. The result that we will employ is stated as below.

\begin{theorem}[Corollary 1.5 in \cite{SheridanFano}]\label{theorem_corsheridan}
The $w$-summand $\mcal{F}uk_w(X)$ of the monotone Fukaya category is trivial if $w$ is not an eigenvalue of the quantum cup product $c_1 (TX) \star$ in $QH(X)$.
\end{theorem}

\subsection{Maslov index formula for gradient holomorphic disks}

We recall the Maslov index formula for gradient holomorphic disks derived in \cite{ChoKimMONO}. Consider a symplectic manifold $X$ equipped with a Hamiltonian ${S}^1$-action. Fix a moment map $H$ of the circle action. Suppose that we have an ${S}^1$-invariant Lagrangian submanifold $L$ contained in a level set of $H$.
Now, we choose the following data$\colon$
\begin{enumerate}
\item an ${S}^1$-invariant compatible almost complex structure $J$ and
\item a free ${S}^1$-orbit $\vartheta$ in $L$.
\end{enumerate}

The moment map $H$ generates the gradient vector field $\nabla H = J X_H$ using the Riemannian metric $g_J$ induced by $\omega$ and $J$ where $X_H$ is the Hamiltonian vector field generated by $H$. Assume that we have the integral curve $\gamma$ of $\nabla H$ starting at $p \in \vartheta$ which converges to a fixed point. The collection of $S^1$-orbits of $\gamma$ forms a $J$-holomorphic disk whose boundary is $\vartheta$. We call such a holomorphic disk a \emph{gradient holomorphic disk} of $\vartheta$, which is an analogue object of holomorphic spheres in Karshon \cite{KarshonPHF}.

\begin{figure}[h]
	\scalebox{0.85}{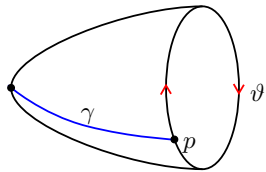}
	\caption{\label{Fig_diagram_AF23467} Gradient holomorphic disk}	
\end{figure}

For a connected Lie group $G$ acting on a manifold $X$, we call the action \emph{semi-free} if the action is free on the complement of the set of fixed points. In the following circumstance, the Maslov index of a gradient holomorphic disk  of $\vartheta$ can be calculated by the codimension of the maximal fixed component of $H$. 

\begin{proposition}[Corollary 3.8 in \cite{ChoKimMONO}]\label{corollary_Maslov_index_formula}
Let $(X,\omega)$ be a $2n$-dimensional symplectic manifold equipped with a Hamiltonian ${S}^1$-action. Choose a Hamiltonian function $H$ of the $S^1$-action. Let $L$ be an ${S}^1$-invariant Lagrangian submanifold of $X$ lying on some level set of a moment map $H$. Suppose that a class $\beta \in \pi_2(X, L)$ is represented by a gradient holomorphic disk of a free ${S}^1$-orbit $\vartheta$ in $L$. If the action is semi-free near the fixed point of the gradient disk and $H$ attains the maximum at the fixed point, then the Maslov index $\mu(\beta)$ is the real codimension of the maximal fixed component.	
\end{proposition}

We now apply Proposition~\ref{corollary_Maslov_index_formula} to prove that the GZ fiber at the ``center" of the GZ polytope $\Delta^n_\lambda$ in~\eqref{equ_GZpolytope} is a monotone Lagrangian torus. Let us choose $\lambda_1 = n$ in the tuple $\lambda$ in~\eqref{equ_givenlambda}. The following point is said to be the \emph{center} of $\Delta_\lambda^n \colon$
$$
\mathbf{u}_\mathrm{0} = (u_1 = 0, u_2 = 1, \dots, u_n = n-1).
$$ 
For our convenience, let us label the inverse image of certain faces of $\Delta_\lambda^n$.
\begin{enumerate}
\item $F_0$ is the inverse image of the facet $f_0$ supported by $u_2 + u_1 = 0$. 
\item $F_i$ is the inverse image of the facet $f_i$ supported by $u_{i+1} - u_{i} = 0$ for $i = 1, \dots, n$.
\item $F_{01}$ is the inverse image of the intersection $f_{01}$ of $f_0$ and $f_1$.
\end{enumerate}

\begin{proposition}\label{cor_monotonetorus}
The Gelfand--Zeitlin fiber $(\Phi^n_\lambda)^{-1}(\mathbf{u}_\mathrm{0})$ is a monotone Lagrangian torus.
\end{proposition}

\begin{proof}
At the center $\mathbf{u}_\mathrm{0}$, the map $\Phi^n_\lambda$ generates a Hamiltonian torus action by Theorem~\ref{theorem_GS}. Thus, the fiber $(\Phi^n_\lambda)^{-1}(\mathbf{u}_\mathrm{0})$ is a free $T^n$-orbit, which is a Lagrangian torus. The quadric $\mcal{Q}_n$ is a Fano manifold so that it is simply connected.  When $n \geq 3$, by the long exact sequence of homotopy groups and Lemma~\ref{lemma_pi2}, we have 
$$
\pi_2 (\mcal{O}^n_\lambda, (\Phi^n_\lambda)^{-1}(\mathbf{u}_\mathrm{0})) \simeq \Z^{n+1}.
$$ 

We shall generate gradient holomorphic disks intersecting $F_i$ using the Hamiltonian ${S}^1$-actions. Set $\Phi_{n+1} := n$. For each $i = 1, \dots, n$, the function $\Phi_{i} - \Phi_{i+1}$ generates a Hamiltonian ${S}^1$-action by Theorem~\ref{theorem_GS} and hence produces a gradient holomorphic disk $\varphi_i$ of an orbit of this $S^1$-action. Because of the pattern~\eqref{equ_GZpolytope}, the maximum component of $\Phi_{i} - \Phi_{i+1}$ occurs over the facet $F_i$ and therefore each class $[\varphi_i]$ has the Maslov index two by Proposition~\ref{corollary_Maslov_index_formula}. Moreover, the symplectic area of the gradient disk is given by 
\begin{equation}\label{equ_areadis}
\int_{\mathbb{D}^2} \varphi_i^* \, \omega_\lambda^n =  (\Phi_{i} - \Phi_{i+1})(\mathbf{u}_{\max}) - (\Phi_{i} - \Phi_{i+1})(\mathbf{u}_0) = 0 - ( - 1) = 1
\end{equation}
where $\mathbf{u}_{\max} \in F_i$.
Also, we generate a gradient holomorphic disk $\varphi_0$ of Maslov index two and area one using the $S^1$-action generated by $- \Phi_1 - \Phi_2$  intersecting $F_{0}$ over the facet $f_0$. In sum,  
$$
\omega([\varphi_i]) = 1 = 2/2 = \mu([\varphi_i]) /2
$$ 
for $i = 0, 1, \dots, n$.
As the classes $[\varphi_0], [\varphi_1], \dots, [\varphi_n]$ form a basis for $\pi_2 (\mcal{O}^n_\lambda, (\Phi^n_\lambda)^{-1}(\mathbf{u}_\mathrm{0}))$, the fiber $(\Phi^n_\lambda)^{-1}(\mathbf{u}_\mathrm{0})$ is monotone.

When $n = 2$, we have $\pi_2 (\mcal{O}^2_\lambda, (\Phi^2_\lambda)^{-1}(\mathbf{u}_\mathrm{0})) \simeq \Z^{4}$. Via the toric degeneration $\Psi_{1,0}$ in Proposition~\ref{proposition_gradientflow}, the Lagrangian sphere $S^2$ located at  the origin $(u_1 = 0, u_2 = 0)$ collapses to a point. There are three gradient disks $\varphi_0, \varphi_1$, and $\varphi_2$ of Maslov index two with symplectic area one intersecting $F_0, F_1$, and $F_2$, respectively.
The classes $[\varphi_0], [\varphi_1], [\varphi_2]$, and $[S^2]$ form a basis for $\pi_2 (\mcal{O}^2_\lambda, (\Phi^2_\lambda)^{-1}(\mathbf{u}_\mathrm{0}))$.  Since the Chern number for Lagrangian sphere is zero and the area is zero, the monotonicity of the fiber $(\Phi^2_\lambda)^{-1}(\mathbf{u}_\mathrm{0})$ follows.
\end{proof}

For later usage, we denote by $\beta_i$ the class represented the gradient disk $\varphi_i$ in the proof of Proposition~\ref{cor_monotonetorus}, intersecting the inverse image of the relative interior of $f_i$ once. That is,  
\begin{equation}\label{equ_betaivarphi}
\beta_i := [\varphi_i]
\end{equation}
for $i = 0, 1, \dots, n$.

\section{Disk potential functions for quadrics}

In this section, we present the main result of this paper and discuss relations with other works. 

\subsection{Disk potential functions for Gelfand--Zeitlin systems of quadrics}

When it comes to the Gelfand--Zeitlin system on every coadjoint $\mathrm{U}(n)$-orbit, the holomorphic disks of Maslov index two can be immediately classified via the polyhedral geometry of the corresponding GZ polytope. Specifically, holomorphic disks of Maslov index two bounded by its GZ torus fiber are those intersecting a facet exactly once according to \cite{NishinouNoharaUeda} as in the toric moment maps of Fano toric manifolds (cf. \cite{ChoOh}). In particular, the terms of the disk potential for the GZ system on a coadjoint $\mathrm{U}(n)$-orbit are in one-to-one correspondence with the facets of the GZ polytope. 

For a coadjoint $\mathrm{SO}(n)$-orbit, \emph{however}, the number of terms of the potential function for the GZ system can be \emph{greater} than that of facets of the GZ polytope. Besides the disks emanating from the inverse image of a facet, there \emph{does} exist a holomorphic disk of Maslov index two  emanating from the inverse image of a certain lower stratum. 

To illustrate this feature, consider a gradient holomorphic disk of an orbit $\vartheta$ generated by $\Phi_2$ in $\mcal{Q}_n \simeq \mathrm{OG}(1, \C^{n+2})$. The maximal component of $\Phi_2$ occurs over the stratum given by $u_2 = n$, which is of codimension $n-1$ in the base $\Delta_\lambda^n$. According to Proposition~\ref{corollary_Maslov_index_formula}, the Maslov index of the class represented by the gradient holomorphic disk is $2(n-1)$. By the adjunction, the minimal Chern number of $\mcal{Q}_n \subset \CP^{n+1}$ is $n$. If there exists a holomorphic disk whose boundary is the orbit $\vartheta^{-1}$, then it should have the Maslov index $2$, so that the topological sum of the disk with boundary $\vartheta^{-1}$ and the disk with boundary $\vartheta$ can have the Chern number $n$. It is supposedly intersecting the inverse image of the codimension two stratum supported by $u_2 = 0$ in Section~\ref{subsection_cominsysonquad}. Due to the additional dimension contribution from the spherical factors on the GZ fibers on the stratum, the codimension of the inverse image of $u_2 = 0$ is two so that the Maslov index of disk with boundary $\vartheta^{-1}$ intersecting this component may have Maslov index two. We shall show that there exists such a holomorphic disk and compute its open Gromov--Witten invariant in Section~\ref{section_computingdiscpotfuncs}.

Let $\vartheta_j$ be an $S^1$-orbit generated by the Hamiltonian $\Phi_j$ in~\eqref{equ_gssys}. We then take holonomy variables
\begin{equation}\label{equ_holonomyvariables}
z_j := \mathrm{hol}_\nabla (\vartheta_j) \in \C^* \mbox{ for $j = 1, \dots, n$} \
\end{equation}
The main theorem of the paper is stated below. 

\begin{theorem}\label{theorem_main}
Assume that $n \geq 2$. The disk potential of the monotone Gelfand--Zeitlin fiber $L := (\Phi^n_\lambda)^{-1}(\mathbf{u}_\mathrm{0})$ is
\begin{equation}\label{equ_potentialGZequngeq3}
W_{L} (\mathbf{z}) = \frac{1}{z_n} + \frac{z_n}{z_{n-1}} + \dots + \frac{z_3}{z_2} + \frac{z_{2}}{z_{1}} + 2 z_2 + z_1 z_2.
\end{equation}
\end{theorem}

The proof of Theorem~\ref{theorem_main} will be provided at the end of Section~\ref{section_computingdiscpotfuncs}. Throughout Section~\ref{sect_effectivediscclasses} and~\ref{section_computingdiscpotfuncs}, we shall prepare lemmas and propositions for the proof of  Theorem~\ref{theorem_main}. Specifically, Section~\ref{sect_effectivediscclasses} is reserved for the classification of all homotopy classes which can be represented by a holomorphic disk of Maslov index two and hence the Laurent monomials of the disk potential $W_L(\mathbf{z})$ in~\eqref{equ_potentialGZequngeq3} will be determined. In Section~\ref{section_computingdiscpotfuncs}, we will compute the counting invariants and therefore the coefficients of the disk potential $W_L(\mathbf{z})$ in~\eqref{equ_potentialGZequngeq3} will be determined. In the remaining part of this section, an example and corollaries of Theorem~\ref{theorem_main} will be discussed.

\begin{example}
When $n = 3$, the potential function is
$$
W_{L}(\mathbf{z}) = \frac{1}{z_3} + \frac{z_3}{z_2} + \frac{z_2}{z_1} + 2 z_2 + {z_1z_2}.
$$
The corresponding holomorphic disks are illustrated in Figure~\ref{fig_holomorphicdiscofMaslovindextwo}.
\begin{figure}[h]
	\scalebox{0.85}{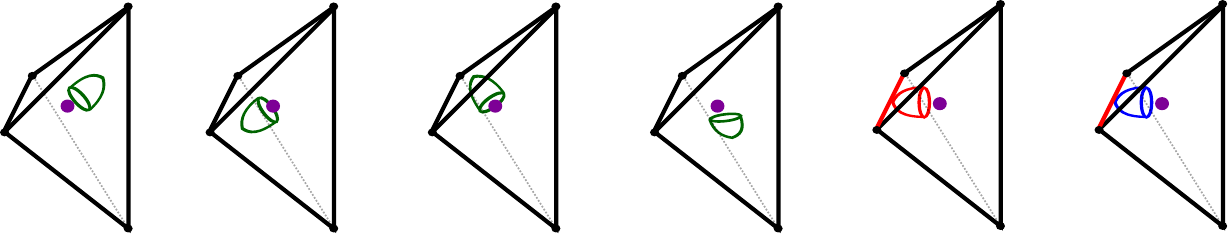}
	\caption{\label{fig_holomorphicdiscofMaslovindextwo} Holomorphic disks of Maslov index two bounded by $L$ in $\mcal{Q}_3$}	
\end{figure}
\end{example}

The disk potential $W_{L}$ can be used to compute the deformed Floer cohomology of $L$. By computing its critical point, the Floer cohomology deformed by $\nabla$ corresponding to the critical point is non-zero so that $L$ is non-displaceable. Recall that $L$ is \emph{non-displaceable} if $L \cap \phi(L) \neq \emptyset$ for any time-dependent Hamiltonian diffeomorphism $\phi$.

\begin{corollary}\label{corollary_nondisplaceable}
The monotone Lagrangian torus $L$ is non-displaceable.
\end{corollary}

\begin{remark}\label{remarkn2n3differ}
The disk potential~\eqref{equ_potentialGZequngeq3} should be interpreted differently when $n = 2$ and $n \geq 3$. When $n = 2$, the term $2 z_2$ arises from two different classes of Maslov index two, which differ by the Lagrangian sphere $[S^2]$ (cf. \cite{FOOOS2S2}). If $n \geq 3$, then the term $2 z_2$ comes from one single class of Maslov index two. Recall from Section~\ref{subsection_cominsysonquad} that the fiber over any point in the relative interior of the stratum given by $u_2 = 0$ of codimension two is diffeomorphic to $S^2 \times T^{n-2}$. When $n \geq 3$, the class represented by a sphere is trivial since it can be isotoped to a sphere class in the fiber $S^n$ at the origin along a line segment in the face supported by $u_2 = 0$.
\end{remark}

More generally, suppose that $L(\mathbf{u}) := (\Phi^n_\lambda)^{-1}(\mathbf{u})$ for an interior point $\mathbf{u}$ (not necessarily being the center) of the GZ polytope $\Delta^n_\lambda$. The fiber $L(\mathbf{u})$ is still a Lagrangian torus, but is no longer monotone {unless u is $u_0$}. The counting invariants are \emph{not} invariant under a choice of almost complex structures in general. Nevertheless, assuming that $\mathbf{u}$ is sufficiently close to $\mathbf{u}_0$, the disk potential function can be defined and has the following form$\colon$ 
\begin{equation}\label{equ_potentialGZequngeq32}
W_{L(\mathbf{u})} (\mathbf{z}) = \frac{1}{z_n} T^{\lambda_1 - u_n} + \frac{z_n}{z_{n-1}} T^{u_n - u_{n-1}} + \dots + \frac{z_3}{z_2} T^{u_3-u_2} + \frac{z_{2}}{z_{1}} T^{u_2 -u_1} + 2 z_2 T^{u_2} + z_1 z_2 T^{u_2 + u_1}.
\end{equation}
It is because by taking a line segment connecting $\mathbf{u}$ and $\mathbf{u}_0$, one can take an isotopy between two Lagrangian tori. The appearance of holomorphic disks of Maslov index $\leq 0$ is a closed condition and the Fredholm regularity is an open condition. If $\mathbf{u}$ and $\mathbf{u}_0$ are closed enough, one constructs a cobordism between moduli spaces of disks.

\begin{remark}
If $\mathbf{u} = \mathbf{u}_0$, that is $L(\mathbf{u})$ is monotone, then the symplectic areas of Maslov index two classes are all equal. For this reason, as in Definition~\ref{def_diskpotentfunct}, the formal parameter $T$ recording the symplectic area is omitted. But, if $L(\mathbf{u})$ is not monotone, the formal parameter $T$ is necessary in~\eqref{equ_potentialGZequngeq32}.
\end{remark}

\subsection{Relation with other works}

When $n = 2$, the disk potential of the form~\eqref{equ_potentialGZequngeq3} for the quadric hypersurface $\mcal{Q}_2 \simeq \CP^1 \times \CP^1$ is known. In \cite{AurouxTdual}, Auroux constructed a Lagrangian torus fibration and discussed the disk potentials in two different chambers to exhibit the wall-crossing phenomenon. The disk potential of the Chekanov torus agrees with~\eqref{equ_potentialGZequngeq3} up to a certain coordinate change. Also, to produce a continuum of non-displaceable tori,  Fukaya, Oh, Ohta, and Ono \cite{FOOOS2S2} constructed Lagrangian tori by doing a surgery replacing the $A_1$-singularity by the Milnor fiber and computed their disk potential function. It agrees with the expression~\eqref{equ_potentialGZequngeq3} up to a certain coordinate change. By the work of Oakley--Usher \cite{OakleyUsher}, the above two monotone Lagrangian tori are related by a Hamiltonian isotopy. The computation result of the disk potential in~\eqref{equ_potentialGZequngeq3} alludes to the Hamiltonian invariance of the Chekanov torus and the monotone Gelfand--Zeitlin torus fiber. The author hopes that the invariance will be discussed in future somewhere else.

We recall that Przhiyalkovski\u{\i}'s mirror, see \cite[Page 775]{Przhiya}
\begin{equation}\label{equ_przmirrorr}
\check{X}_\mathrm{Prz} := (\C^*)^{n},  \quad W_\mathrm{Prz} (x_1, \dots, x_n) = x_1 + \dots + x_{n-1} + \frac{(x_n + 1)^2}{x_1 x_2 \dots x_n}.
\end{equation}
The LG mirror derived from disk counting in Theorem~\ref{theorem_main} agrees with one cluster chart of Pech--Rietsch--Williams mirror and the superpotential restricted to the chart in Section~\ref{mirrorsymmofquadricss}. In particular, the disk potential~\eqref{equ_potentialGZequngeq3} is equal to Przhiyalkovski\u{\i}'s superpotential under the coordinate change
\begin{equation}\label{equ_coordchaprzgz}
x_1 \mapsto \frac{1}{z_n}, \,\, x_n \mapsto z_1, \,\, x_i \mapsto \frac{z_{n+2-i}}{z_{n+1-i}} \quad \mbox{for $i = 2, \cdots, n-1$}.
\end{equation}
Namely, counting of stable disks bounded by the GZ torus fiber produces Przhiyalkovski\u{\i}'s mirror. 

\begin{corollary}\label{corollary_Pzrmirror}
The Przhiyalkovski\u{\i}'s superpotential of $\mcal{Q}_n$ is equal to the disk potential function of Gelfand--Zeitlin torus fiber under the coordinate change~\eqref{equ_coordchaprzgz}.
\end{corollary}

Homological mirror symmetry of quadric hypersurfaces (more generally Fano hypersurfaces) was established by Sheridan \cite{SheridanFano}. He employed an immersed Lagrangian with bounding cochains to split-generate the monotone Fukaya category and the non-zero objects were classified into two groups$\colon$ objects corresponding to a small eigenvalue and objects corresponding to a big eigenvalue. The monotone torus $L$ in Theorem~\ref{theorem_main} together with $\nabla$ corresponding to a critical point is a non-zero object with small eigenvalue. Moreover, by the structural result in \cite{SheridanFano}, each non-zero object $(L, \nabla)$ split-generates one summand of the quantum cohomology with a non-zero eigenvalue. 

The disk potential function of $\mathrm{OG}(1, \C^5)$ was computed in \cite{HongKimLau}. They have used the Lefschetz fibration model in order to construct a singular Lagrangian and to compute open Gromov--Witten invariants by the wall-crossing. As in $\mathrm{OG}(1, \C^5)$, the deformation of an immersed Lagrangian was computed to compactify the LG mirror partially. Since we are playing with a completely integrable system, one needs a different strategy to compute open Gromov--Witten invariants.  Also, it would be interesting to construct immersed Lagrangians to compactify the LG mirror of Lagrangian tori partially for arbitrary quadric hypersurface.

\section{Effective disk classes}\label{sect_effectivediscclasses}

Let $\mcal{X} = \{ \mcal{X}_{\mathbf{t}} \}$ be the toric degeneration of the quadric $\mcal{Q}_n$ in~\eqref{equ_familyofquadrics}. For notational simplicity, let us set $\mcal{X}_{t} := \mcal{X}_{(t)}$ and recall that $\Psi_{1,t} \colon \mcal{X}_{1} \to \mcal{X}_{t}$ is the map  in Proposition~\ref{proposition_gradientflow}.
We denote by $L_1$ the monotone GZ torus fiber $(\Phi^n_\lambda)^{-1}(\mathbf{u}_\mathrm{0})$ in Proposition~\ref{cor_monotonetorus} and set 
$L_t :=  \Psi_{1,t}(L_1)$ in $\mcal{X}_t$.

The main proposition in this section characterizes the homotopy classes of Maslov index two which can be represented by a holomorphic disk in terms of the intersection with the inverse images of faces. It greatly reduces homotopy classes which can contribute to the disk potential of $L_1$.

\begin{proposition}\label{proposition_charmaslovindextwo}
If $\beta \in \pi_2(\mcal{X}_1, L_1) $ is represented by a $J$-holomorphic map $\varphi_\beta$ of Maslov index two bounded by $L_1$ for some generic compatible almost complex structure $J$ with $n_\beta \neq 0$, then $\varphi_\beta$ intersects the inverse image over the relative interior of either a facet or the stratum $f_{01}$, supported by $u_2 = 0$ once. 
\end{proposition}

To prove Proposition~\ref{proposition_charmaslovindextwo}, we exploit the toric degeneration and its toric blowup $\widehat{\mcal{X}}_0$, which is described below.  Let $\Sigma$ be the fan in $N_\R$ associated to the toric variety $\mcal{X}_0$. Taking the primitive vector $v := (0, 1, 0, \dots, 0)$, let $\widehat{\mcal{X}}_0$ be the toric variety associated to the star subdivision $\widehat{\Sigma}:= \Sigma^*(v)$ of $\Sigma$ at $v$. Dually, consider a polytope $\widehat{\Delta} := \widehat{\Delta^n_\lambda}$ in $M_\R$ obtained from the GZ polytope $\Delta^n_\lambda$ by chopping the singular stratum $f_{01}$. That is, the chopped polytope $\widehat{\Delta}$ is determined by the intersection of $\Delta^n_\lambda$ and $u_{2} - \varepsilon \geq 0$ for a positive real number $\varepsilon$ sufficiently close to $0$. Then $\widehat{\mcal{X}}_0$ is a smooth toric variety, the toric blowup along the singular stratum $\Psi_{1} (F_{01})$ where $\Psi_{1} := \Psi_{1,0}$. Let us keep in mind the following diagram throughout this section$\colon$ 
\begin{equation}\label{equ_toricdegandresolu}
\xymatrix{
& (\widehat{\mcal{X}}_0, \widehat{L}_0) \ar[d]^{\Pi} \\
(\mcal{X}_1, L_1) \ar[r]_{\Psi_{1}} & (\mcal{X}_0, L_0)
}
\end{equation}
where $\Pi$ is the toric blow-up map. The inverse image of the interior $\mathring{\Delta}^n_\lambda$ of the GZ polytope ${\Delta}^n_\lambda$ can be viewed as the cotangent bundle of $L_1$ topologically. 

\begin{remark}
Toric degenerations had been employed to compute disk potentials in \cite{NishinouNoharaUeda} when the central toric varieties admit small toric resolutions and each algebraic variety of the family is Fano. We emphasize that our situation does \emph{not} fit into their situation. 
\end{remark}

The first lemma asserts that every holomorphic disk in a \emph{non-trivial} class must intersect the inverse image of a lower dimensional stratum of ${\Delta}^n_\lambda$. 

\begin{lemma}\label{lemma_maslovinter}
Every pseudo-holomorphic disk $\varphi_\beta$ in a non-trivial class  $\beta \in \pi_2(\mcal{X}_1, L_1)$ must intersect the inverse image of $\Delta_\lambda^n - \mathring{\Delta}^n_\lambda$.
\end{lemma}

Let us recall one fact on the Maslov index of a map into a cotangent bundle.

\begin{lemma}[Lemma 9.2 in \cite{Cho}]\label{prop_Maslovzero}
Let $L$ be a Lagrangian submanifold such that its tangent bundle $TL$ is trivial. Let $o_L$ be the zero section of the cotangent bundle $T^*L$.
Then every smooth map from a boarded Riemann surface $(\Sigma, \partial \Sigma)$ into $(T^*L, o_L)$ has the Maslov index zero. 
\end{lemma}

\begin{proof}[Proof of Lemma~\ref{lemma_maslovinter}]
Suppose that $\varphi_\beta$ is contained in the cotangent bundle $T^*L_1$. By Lemma~\ref{prop_Maslovzero}, 
the Maslov index of $\beta$ is zero. Then the monotonicity of $L_1$ yields that the class $\beta$ must be trivial because of the positivity of area of pseudo-holomorphic curves. 
\end{proof}

Consider the long exact sequences of homotopy groups and homology groups related by the Hurewicz homomorphisms.
As every Fano manifold is simply connected, the short five lemma yields that the Hurewicz homomorphism is an isomorphism $\pi_2(\mcal{X}_1, L_1) \simeq H_2(\mcal{X}_1, L_1; \Z)$.
\begin{equation}\label{equ_twoexactsequencess}
\xymatrix{
\pi_2 (L_1) \simeq 0 \ar[r] \ar[d]_{0} & \pi_2 (\mcal{X}_1) \ar[r] \ar[d]_{\simeq} & \pi_2 (\mcal{X}_1, L_1) \ar[r] \ar[d] & \pi_1 (L_1) \ar[r] \ar[d]_{\simeq} & 0 \\
H_2 (L_1; \Z) \ar[r] & H_2 (\mcal{X}_1; \Z) \ar[r] &
H_2 (\mcal{X}_1, L_1; \Z) \ar[r] & H_1 (L_1; \Z) \ar[r] & 0}
\end{equation}
From now on, the above groups will be used interchangeably.

Also, the 2nd relative homology group of every toric fiber in the toric orbifold $\mcal{X}_0$ is generated by the holomorphic disks emanated from the toric divisor over the field $\mathbb{Q}$ of rational numbers. (To generates the 2nd relative homology group over $\mathbb{Z}$, in addition to those holomorphic disks, orbi-disks are necessary in general, see \cite[Section 9.1]{ChoPoddar} for more details.) 
For this reason, we prefer using the coefficient field $\mathbb{Q}$ for homology groups$\colon$
$$
\pi_2(\mcal{X}_1, L_1) \otimes \Q \simeq H_2(\mcal{X}_1, L_1; \Z)  \otimes \Q \simeq H_2(\mcal{X}_1, L_1; \Q)
$$ 
by the universal coefficient theorem.

We analyze the homology groups via the map $\Psi_1 := \Psi_{1,0}$.

\begin{lemma}\label{lemma_gradientdiscflow}
The induced homomorphism 
$$
\begin{cases}
\Psi_{1,*} \colon H_2 (\mcal{X}_1, L_1; \Q) \simeq \Q^{4} \to H_2 (\mcal{X}_0, L_0; \Q) \simeq \Q^{3} \mbox{ is an epimorphism} &\mbox{if $n = 2$,} \\
\Psi_{1,*} \colon H_2 (\mcal{X}_1, L_1; \Q) \simeq \Q^{n+1} \to H_2 (\mcal{X}_0, L_0; \Q) \simeq \Q^{n+1} \mbox{ is an isomorphism} &\mbox{if $n \geq 3$.}
\end{cases}
$$ 
\end{lemma}

\begin{proof}
We first check the statement when $n \geq 3$. Since $\Delta^n_\lambda$ is a simplex with $n+1$ facets, $\mcal{X}_0$ is a toric orbifold with $H_2(\mcal{X}_0, L_0; \Q) \simeq \Q^{n+1}$.  From the fact that $\mcal{X}_1$ is simply connected and Lemma~\ref{lemma_pi2}, it follows that $H_1(\mcal{X}_1; \Z) = 0$ and  $H_2(\mcal{X}_1; \Z) = \Z$. Then the long exact sequence $\mbox{\eqref{equ_twoexactsequencess}}\otimes \Q$ of homology groups yields $H_2(\mcal{X}_1, L_1; \Q) \simeq \Q^{n+1}$. Thus, it suffices to verify that a basis for $H_2(\mcal{X}_1, L_1; \Q)$ sends to that for $H_2(\mcal{X}_0, L_0; \Q)$ via $\Psi_{1,*}$. The map $\Psi_1$ is a symplectomorphism over the complement of $F_{01}$. The complement contains \emph{all} gradient disks $\varphi_0, \varphi_1, \dots, \varphi_n$ in Proposition~\ref{cor_monotonetorus}, emanating from the inverse image of a facet. Under the map $\Psi_1$, those disks map into disks intersecting a toric divisor, preserving the intersection numbers between disks and divisors. Those disks generate $H_2(\mcal{X}_0, L_0; \Q)$. It completes the proof of Lemma~\ref{lemma_gradientdiscflow}.

In this case where $\mcal{X}_1 \simeq \CP^1 \times \CP^1$, the fiber $(\Phi^2_{\lambda})^{-1}(\mathbf{0})$ over the origin is a Lagrangian sphere, which collapses to the orbifold point through the map $\Psi_1$. Hence, the class represented by the Lagrangian sphere maps into the trivial class via $\Psi_{1,*}$. Again by analyzing three disks emanating from a facet, we have verified that the map $\Psi_{1,*} \colon H_2 (\mcal{X}_1, L_1; \Q) \simeq \Q^4 \to H_2 (\mcal{X}_0, L_0; \Q) \simeq \Q^3$ is surjective. 
\end{proof}

Consider the toric blow-up $\widehat{\mcal{X}}_0$ of $\mcal{X}_0$ along the singular stratum over $f_{01}$ in~\eqref{equ_toricdegandresolu}. The fan $\widehat{\Sigma}= \Sigma^*(v)$ associated to $\widehat{\mcal{X}}_0$ has $(n+2)$ one cones, i.e. the toric anti-canonical divisor has $(n+2)$ irreducible components. Let us denote by $\widehat{F}_{01}$ the exceptional divisor of $\widehat{\mcal{X}}_0$, which contracts to the singular stratum under the toric blow-up map $\Pi$ in~\eqref{equ_toricdegandresolu}. Also, we denote by $\widehat{F}_{i}$ the strict transform of ${F}_{i}$ in $\widehat{\mcal{X}}_0$ for $i = 0, 1, \dots, n$. 

We equip $\widehat{\mcal{X}}_0$ with a torus-invariant symplectic form associated to the polytope $\widehat{\Delta}$ obtained by chopping the codimension two stratum $f_{01}$ small enough. As the map $\Pi$ is torus-equivariant, $\Pi$ provides a correspondence between the toric fibers of $\widehat{\mcal{X}}_0$ and the toric fibers of ${\mcal{X}}_0$. We denote by $\widehat{L}_0$ by the lift of $L_0$ via $\Pi$. Although the map does \emph{not} preserve the symplectic areas of holomorphic disks bounded by $L_0$ and $\widehat{L}_0$, we will be only concerned with topological invariants such as Maslov indices and thus $\Pi$ is good enough for our purpose.

According to the classification of holomorphic disks bounded by toric fibers by Cho--Oh \cite{ChoOh}, $\widehat{L}_0$ bounds $(n+2)$ many holomorphic disks without sphere bubbles (which are often called \emph{basic disks}) of Maslov index two. Each holomorphic disk of Maslov index two intersects one of the irreducible components of the toric anti-canonical divisor. Let us denote by $\widehat{\beta}_\bullet$ the class represented by a holomorphic disk intersecting $\widehat{F}_\bullet$ for $\bullet = 01, 0, 1, \dots, n$. Those classes generate $H_2 (\widehat{\mcal{X}}_0, \widehat{L}_0; \Z) \simeq \Z^{n+2}$ as $\widehat{\Sigma}$ has $(n+2)$ one cones, see \cite[(1.7)]{FOOOToric2}. Thus, because of the discrepancy on the dimensions of the domain and the codomain, $\Pi_* $ is \emph{not} injective.

Note that the class $\widehat{\beta}_1 + \widehat{\beta}_0 - 2 \widehat{\beta}_{01} \in H_2 (\widehat{\mcal{X}}_0, \widehat{L}_0; \Z)$ maps down to the trivial class in $H_2 ({\mcal{X}}_0, L_0; \Z)$ along $\Pi_*$. Also, its Maslov index is zero since the Maslov indices of $\widehat{\beta}_{1}, \widehat{\beta}_{0}$, and $\widehat{\beta}_{01}$ are all equal to two. Consider the quotient group $\overline{H}_2 (\widehat{\mcal{X}}_0, \widehat{L}_0; \Z)$ of $H_2 (\widehat{\mcal{X}}_0, \widehat{L}_0; \Z)$ by the subgroup generated by $\widehat{\beta}_1 + \widehat{\beta}_0 - 2 \widehat{\beta}_{10}$. That is,
$$
\overline{H}_2 (\widehat{\mcal{X}}_0, \widehat{L}_0; \Z) = H_2 (\widehat{\mcal{X}}_0, \widehat{L}_0; \Z) / \langle \widehat{\beta}_1 + \widehat{\beta}_0 - 2 \widehat{\beta}_{10} \rangle.
$$
Then the induced map $\Pi_*$ factors through $\overline{H}_2 (\widehat{\mcal{X}}_0, \widehat{L}_0; \Z) \colon$
$$
\xymatrix{
{H}_2 (\widehat{\mcal{X}}_0, \widehat{L}_0; \Z) \ar[dr] \ar[rr]^{\Pi_*}  & & {H}_2 ({\mcal{X}}_0, {L}_0; \Z) \\
& \overline{H}_2 (\widehat{\mcal{X}}_0, \widehat{L}_0; \Z) \ar[ur]_{\overline{\Pi}_*}& 
}
$$
By the universal coefficient theorem, we then obtain an isomorphism over $\Q \colon$
\begin{equation}\label{equ_projectionblowup}
\overline{\Pi}_*:= \overline{\Pi}_* \otimes \Q \colon  \overline{H}_2 (\widehat{\mcal{X}}_0, \widehat{L}_0; \Q) \to H_2 (\mcal{X}_0, L_0; \Q).
\end{equation}
Here, by abuse of notation, the induced map over $\Q$ will be denoted by $\overline{\Pi}_*$.

The Maslov homomorphism descends from ${H}_2 (\widehat{\mcal{X}}_0, \widehat{L}_0; \Z)$ to $\overline{H}_2 (\widehat{\mcal{X}}_0, \widehat{L}_0; \Z)$ since the Maslov index of $\widehat{\beta}_1 + \widehat{\beta}_0 - 2 \widehat{\beta}_{01}$ is zero. In other words, any representatives of a class in $\overline{H}_2 (\widehat{\mcal{X}}_0, \widehat{L}_0; \Z)$ have the same Maslov index. Also, the Maslov homomorphism over $\Q$ can be formally defined by extending the Maslov homomorphism over $\Z$ linearly. Moreover, the map~\eqref{equ_projectionblowup} preserves the Maslov index because the Maslov indices of $\widehat{\beta}_i$ and $\overline{\Pi}_* (\widehat{\beta}_i)$ are equal for $i = 0, 1, \dots, n$.

If $n \geq 3$, then let $\widehat{\Psi}_{1,*}$ be the composition $\overline{\Pi}^{-1}_* \circ \Psi_{1,*}$ and hence 
\begin{equation}\label{equ_compositionmap}
\widehat{\Psi}_{1,*} \colon H_2 (\mcal{X}_1, L_1; \Q) \to H_2 (\mcal{X}_0, L_0; \Q) \to  \overline{H}_2 (\widehat{\mcal{X}}_0, \widehat{L}_0; \Q)
\end{equation}
is an isomorphism by Lemma~\ref{lemma_gradientdiscflow}. 
When $n = 2$, the composition~\eqref{equ_compositionmap} is an epimorphism. In this case, the Lagrangian $(\Phi^2_\lambda)^{-1}(\mathbf{0})$ vanishes. Since the Chern number of the vanishing cycle is zero, the composition still commutes with the Maslov homomorphism.

\begin{lemma}\label{lemma_3maslovcorr}
For every $\beta \in H_2(\mcal{X}_1, L_1; \Q)$, the Maslov index of $\beta$ is equal to that of $\widehat{\Psi}_{1,*} (\beta)$.
\end{lemma}

\begin{proof}
For $n \geq 3$, the map~\eqref{equ_compositionmap} is an isomorphism. It suffices to show that the Maslov indices of the classes which form a basis for $H_2 (\mcal{X}_1, L_1; \Q)$ are preserved through~\eqref{equ_compositionmap}. For each $i = 0, 1, \dots, n$, we take a holomorphic disk in $\Psi_{1,*}(\beta_i)$ intersecting the toric divisor over the facet $F_i$ where $\beta_i$ is in~\eqref{equ_betaivarphi}. Consider its strict transformation, which is in the class $\widehat{\beta}_i$ in ${H}_2 (\widehat{\mcal{X}}_0, \widehat{L}_0; \Q)$. The Maslov index of $\beta_i$ is equal to that of $\widehat{\beta}_i$. Since $\widehat{\beta}_i = \widehat{\Psi}_{1,*}(\beta_i)$ in $ \overline{H}_2 (\widehat{\mcal{X}}_0, \widehat{L}_0; \Q)$, the Maslov indices are preserved via the map~\eqref{equ_compositionmap}. 

When $n = 2$, the vanishing cycle (the Lagrangian two sphere $(\Phi^2_{\lambda^\prime})^{-1}(\mathbf{0})$) has the Chern number zero. The same argument for the disks intersecting the inverse image of a facet verifies Lemma~\ref{lemma_3maslovcorr}.
\end{proof}

\begin{corollary}\label{corollary_twodiscs}
For $n \geq 2$, suppose that two classes $\beta$ and $\beta^\prime \in H_2(\mcal{X}_1, L_1; \Q)$ satisfy $\Psi_{1,*}(\beta) = \Psi_{1,*}(\beta^\prime)$. Then $\pa \beta = \pa \beta^\prime$ in $H_1(L_1; \Q)$.
\end{corollary}

\begin{proof}
Note that $\Psi_{1,*}(\beta) = \Psi_{1,*}(\beta^\prime)$ yields $\partial \Psi_{1,*}(\beta) =  \partial \Psi_{1,*}(\beta^\prime)$ in $H_1(L_0)$. The naturality of the long exact sequences yields the following diagram
$$
\xymatrix{
H_2 (\mcal{X}_1, L_1; \Q) \ar[r]^{\Psi_{1,*}} \ar[d]_\partial & H_2 (\mcal{X}_0, L_0; \Q) \ar[d]^\partial  \\
H_1 (L_1; \Q) \ar[r]_{ \Psi_{1,*}}  & H_1 (L_0; \Q), }
$$
which asserts that $\Psi_{1,*}(\partial \beta) =  \Psi_{1,*}(\partial \beta^\prime)$. Since $\Psi_{1,*}$ is an isomorphism on $H_1(L_\bullet; \Q)$, we obtain $\pa \beta = \pa \beta^\prime$ in $H_1(L_1; \Q)$.
\end{proof}

To analyze possible classes which can be realized as a \emph{limit} of holomorphic disks of Maslov index two, we need the following proposition which is derived from the classification result and the area formula of holomorphic curves bounded by a toric fiber in a symplectic toric orbifold.

\begin{proposition}[Section 6 $\&$ 7 in \cite{ChoPoddar}]\label{proposition_classification}
For $n \geq 2$, the symplectic area of a non-trivial class $\beta \in H_2(\mcal{X}_0, L_0; \Q)$ which can be represented by a holomorphic curve bounded by $L_0$ is greater than or equal to one.

In particular, if the moduli space of pseudo-holomorphic curves in a non-trivial class $\beta \in H_2(\mcal{X}_0, L_0; \Q)$ consists of multiple components, then the symplectic area of $\beta$ is strictly greater than one.
\end{proposition}

\begin{proof}
The face supported by $u_2 = 0$ of codimension two has the isotropy group $\mathbb{Z}/2\mathbb{Z}$. By the classification result of holomorphic curves bounded by a toric fiber in \cite[Theorem 6.2]{ChoPoddar}, a class represented by a holomorphic (orbi)disk is of the following form$\colon$
$$
\beta = \sum_{j=0}^n c_j \beta_j \quad (c_j \geq 0 \mbox{ and } 2 c_j \in \Z).
$$
where $\beta_j$ was defined in~\eqref{equ_betaivarphi}.
Since $\omega(\beta_j) = 1$ for each $j$ by~\eqref{equ_areadis}, we have
$$
\omega(\beta) = \sum_{j=0}^n c_j.
$$
For each $\beta_j$, each component $\partial \beta_j \in N$ is $0$, $1$, or $-1$.
Every non-trivial component of each $\partial \beta$ is in the lattice $N$ because the boundary of $\beta$ forms a loop.  
Therefore, the condition $\partial \beta \in N$ yields that $\omega(\beta) \geq 1$.

Suppose that a holomorphic curve in $\beta$ has a sphere or disk bubble. Then the configuration looks like a disk component together with (sphere or disk) bubbles. The disk component consumes the area at least one and bubbles consumes additional symplectic area. Thus, the area must be strictly greater than one.  
\end{proof}

\begin{lemma}\label{lemma_limitofholomorphicdiscs}
Suppose that $n \geq 2$.
If $\beta$ in $H_2(\mcal{X}_1, L_1)$ is represented by a pseudo-holomorphic map of Maslov index two, then there exists a holomorphic curve $\widehat{\varphi}_\beta$ bounded by $\widehat{L}_0$ such that $\Psi_{1,*} (\beta) = \overline{\Pi}_* ([\widehat{\varphi}_\beta])$.
\end{lemma}

\begin{proof}
Choose a monotonically decreasing sequence $(t_j)_{j \in \mathbb{N}}$ of real numbers converging to zero. Recall that each $\mcal{X}_{t_j}$ has a complex structure $J_{t_j}$. By perturbing them to compatible almost complex structures if necessary, we may assume that each $J_{t_j}$ is generic and $\{ J_{t_j} \}$ converges to the $J_0$. Then a $J_{t_j}$-holomorphic disk $\varphi_{t_j}$ converges to a holomorphic curve $\varphi_0$ representing in class $\Psi_{1,*} (\beta)$ by the Gromov's compactness theorem. Each $J_{t_j}$-holomorphic disk $\varphi_{t_j}$ has Maslov index two and hence area one because $L_{t_j}$ is monotone. Also, the limit $\varphi_0$ of the sequence $\{\varphi_{t_j}\}$  has area one. By Proposition~\ref{proposition_classification}, it does not admit any bubbles and hence intersect the singular loci of $\mcal{X}_0$ at most finitely many points. We can lift $\varphi_0$ into the holomorphic curve $(\widehat{\mcal{X}}_0, \widehat{L}_0)$ without losing any components. Its strict transformation $\widehat{\varphi}_\beta$ of $\varphi_0$ is a desired holomorphic curve. (If $\varphi_0$ were with a sphere bubble in the singular loci, then it would be different from $\Pi \circ \widehat{\varphi}_\beta$ because $\Pi \circ \widehat{\varphi}_\beta$ consists of disk components. Also, $\Psi_{1,*} (\beta) = [\varphi_0] \neq [\Pi \circ \widehat{\varphi}_\beta] = \overline{\Pi}_* ([\widehat{\varphi}_\beta])$.)
\end{proof}
 
Now, we are ready to prove Proposition~\ref{proposition_charmaslovindextwo}. 

\begin{proof}[Proof of Proposition~\ref{proposition_charmaslovindextwo}]
Since $L_1$ is monotone, it does not bound any non-constant disks of Maslov index $\leq 0$. Moreover, the index condition (the minimal Chern number of $\mcal{Q}_n$ is $\geq 2$) prohibits sphere bubbles of stable maps of the moduli space in the class $\beta$. Therefore, we may assume that $\varphi_\beta$ is a $J$-holomorphic disk (without any bubbles). 
By Lemma~\ref{lemma_maslovinter}, $\varphi_\beta$ must intersect the inverse image of $\Delta_\lambda^n - \mathring{\Delta}^n_\lambda$.

By Lemma~\ref{lemma_limitofholomorphicdiscs}, there exists a holomorphic map $\widehat{\varphi}_\beta$ bounded by the toric fiber $\widehat{L}_0$ such that $\Psi_{1,*} (\beta) = \overline{\Pi}_* ([\widehat{\varphi}_\beta])$. By Lemma~\ref{lemma_3maslovcorr}, the Maslov index of $\beta$ is equal to that of $\widehat{\Psi}_{1,*} (\beta) = \overline{\Pi}^{-1}_* (\Psi_{1,*} (\beta))$. Suppose that $\varphi_\beta$ intersects 
\begin{enumerate}
\item the inverse image of $\Delta_\lambda^n - \mathring{\Delta}^n_\lambda$ multiple times.
\item the inverse image $\mathring{F}$ of the relative interior of $f$ satisfying $f \neq f_{01}, f_0, f_{1}, \dots, f_n$. 
\end{enumerate}
In the first case, the holomorphic map $\widehat{\varphi}_\beta$ intersects toric strata multiple times. By the classification of holomorphic disks bounded by a toric fiber in \cite[Theorem 5.1]{ChoOh}, the Maslov index of $\widehat{\varphi}_\beta$ is $\geq 4$. In the latter case, the holomorphic map $\widehat{\varphi}_\beta$ intersects a toric stratum whose codimension is $\geq 2$. Then the Maslov index of $\widehat{\varphi}_\beta$ is $\geq 4$. Hence, in both cases, $\varphi_\beta$ must have a Maslov index $\geq 4$.
\end{proof}

\section{Computing counting invariants}\label{section_computingdiscpotfuncs}

The aim of this section is to complete our calculation of the disk potential of $L_1$. The collection $\{ \vartheta_1, \dots, \vartheta_n \}$ of oriented loops in $L_1$ generated by $\Phi_j$'s in~\eqref{equ_gssys} constitutes a basis for $\pi_1( L_1) \simeq \Z^n$. Recall that $z_j$ is the holonomy variable associated to the loop $\vartheta_j$ in \eqref{equ_holonomyvariables}.

Then the disk potential function $W_{L_1}$ of $L_1$ can be expressed as a Laurent polynomial in terms of $z_1, \dots, z_n$ variables. The following proposition claims that the disk potential has to be of the following form~\eqref{equ_undeterminedpotential}.

\begin{proposition}\label{proposition_potentialfunctionm}
The disk potential function of $L_1$ is of the form$\colon$
\begin{equation}\label{equ_undeterminedpotential}
W_{L_1}(\mathbf{z}) = \frac{1}{z_n} + \frac{z_n}{z_{n-1}} + \dots + \frac{z_3}{z_2} + \frac{z_{2}}{z_{1}} +  \kappa \cdot z_2 +  z_1 z_2
\end{equation}
for some $\kappa \in \Z$.
\end{proposition}

To verify Proposition~\ref{proposition_potentialfunctionm}, we collect some lemmas. For each $i = 0, 1, \dots, n$, recall that $\beta_i$ is the homotopy class represented by a gradient disk intersecting the inverse image of the relative interior of $f_i$, see~\eqref{equ_betaivarphi}. To orient the moduli spaces of stable maps into $(\mcal{X}_1, L_1)$, we take the orientation and the standard spin structure for $L_1$ induced by the (local) torus action. We then compute the open Gromov--Witten invariant of $\beta_i$ for $i = 0,1, \dots, n$. 

\begin{lemma}\label{lemma_computationofogw}
For each class $\beta_i$ $(i= 0, 1, \dots, n)$, the open Gromov--Witten invariant $n_{\beta_i}$ is one. Namely, the degree of the evaluation map $\mathrm{ev}_0 \colon \mcal{M}_1(L_1, \beta_i) \to L_1$ is one.
\end{lemma}

\begin{proof}
Let $J_0$ be the toric complex structure of the toric orbifold $\mcal{X}_0$. Fix an isomorphism $L_0 \simeq (S^1)^n$ determined by the toric moment map $(\Phi_1, \dots, \Phi_n)$. Moreover, it trivializes the tangent bundle of $L_0$ and thus gives an orientation of $L_0$. Also, the spin structure given by this trivialization orients the moduli spaces of stable maps into $\mcal{X}_0$.

For each class $\beta_i$, the class $\Psi_{1,*} (\beta_i)$ is Fredholm regular and the evaluation map is an orientation preserving diffeomorphism
$$
\mathrm{ev}_0 \colon \mcal{M}_1 (L_0, J_0; \Psi_{1,*} (\beta_i)) \to L_0
$$
by the analysis of the moduli spaces of holomorphic disks bounded by a toric fiber in \cite[Proposition 9.3]{ChoPoddar}. By abuse of notation, let us simply denote $\Psi_{1,*} (\beta_i)$ by $\beta_i$

The holomorphic disk $\varphi_{\beta_i}$ bounded by $L_0$ emanating from the toric divisor over $f_i$ is fully contained in the complement of a sufficiently small open subset $\mcal{U}$ of the singular strata over $f_{01}$. 
After pulling the complex structure $J_0$ back to the complement of $\Psi^{-1}_1(\mcal{U})$, extend the pull-backed complex structure to a compatible almost complex structure on $\mcal{X}_1$. 
Let us denote an extended almost complex structure by $J_1$.  
By our choice of $J_1$, we then have the following commutative diagram consisting of orientation preserving diffeomorphisms
$$
\xymatrix{
\mcal{M}_1 (L_1, J_1; \beta_i) \ar[r]^{\Psi_{1,*} \,\,\,\,\,\, } \ar[d]_{\mathrm{ev}_0} & \mcal{M}_1 (L_0, J_0;  \beta_i) \ar[d]^{\mathrm{ev}_0}  \\
L_1 \ar[r]_{ \Psi_{1}}  & L_0. }
$$
By Lemma~\ref{lemma_invarianceofcounting}, the counting $n_{\beta_i}$ is invariant under a choice of {generic} compatible almost complex structure and hence $n_{\beta_i} = 1$.
\end{proof}

Recall that the lifted toric fiber of $L_0$ via the toric blow-up $\Pi \colon \widehat{\mcal{X}}_0 \to \mcal{X}_0$ along the singular stratum over $f_{01}$ is denoted by $\widehat{L}_0$. 
The following lemma classifies all homotopy classes of Maslov index two in $\pi_2(\widehat{\mcal{X}}_0, \widehat{L}_0)$ that can be realized as a holomorphic curve bounded in the toric resolution $\widehat{\mcal{X}}_0$.

\begin{lemma}\label{lemma_aaa}
Let $\widehat{\beta}_i$ be the class represented by a holomorphic disk intersecting $\widehat{F}_i$ in $\pi_2(\widehat{\mcal{X}}_0, \widehat{L}_0)$ where $\widehat{F}_i$ is the strict transform of ${F}_i$. 
Every class which can be represented by a holomorphic curve of Maslov index two belongs to one of the following list of homotopy classes$\colon$
\begin{equation}\label{equ_possibleclasse}
\left\{ \widehat{\beta}_0, \widehat{\beta}_1, \dots, \widehat{\beta}_n \right\} \cup \left\{ \widehat{\beta}_{01} + \ell (\widehat{\beta}_0 + \widehat{\beta}_1 - 2 \widehat{\beta}_{01}) \colon \ell \in \mathbb{Z} \right\}.
\end{equation}
\end{lemma}

\begin{proof}
According to \cite[Theorem 11.1]{FOOOToric1}, a homotopy class which can contribute to the disk potential of a toric fiber contains either
\begin{enumerate}
\item (Type I) a holomorphic disk without sphere bubbles or
\item (Type II) holomorphic disks with sphere bubbles
\end{enumerate}
Notice that the classes $\widehat{\beta}_0, \widehat{\beta}_1, \dots, \widehat{\beta}_n$, and $\widehat{\beta}_{01}$ in~\eqref{equ_possibleclasse} belong to (Type I) by the classification of Cho--Oh in~\cite{ChoOh}.

We claim that if there exists a class in (Type II), then the class must be of the form $\widehat{\beta}_{01} + \ell (\widehat{\beta}_0 + \widehat{\beta}_1 - 2 \widehat{\beta}_{01})$ for some $\ell \in \Z - \{0\}$. The GZ polytope $\Delta^n_\lambda$ (up to a translation $u_i^\prime = u_i + (i - 1)$) is reflexive, see \cite[Theorem 3.7]{ChoKimSO} for instance. By chopping the polytope $\Delta^n_\lambda$ along $f_{01}$, we obtain a semi-Fano toric manifold $\widehat{\mcal{X}}_0$, i.e.,  its anti-canonical line bundle is numerically effective. It implies that the Chern number of a non-constant holomorphic sphere is greater or equal to zero. Because of the additivity of Maslov indices and Chern numbers, the only possible configuration of holomorphic curve of Maslov index two bounded by $\widehat{L}_0$ is a holomorphic disk together with a sphere bubble of Chern number zero. 
The holomorphic sphere of Chern number zero can occur only in the exceptional divisor and represents a multiple of $\widehat{\beta}_0 + \widehat{\beta}_1 - 2 \widehat{\beta}_{01}$. Hence, such sphere classes interact only with $\widehat{\beta}_{01}$.
\end{proof}

We are now ready to prove Proposition~\ref{proposition_potentialfunctionm}.

\begin{proof}[Proof of Proposition~\ref{proposition_potentialfunctionm}]
Suppose that there exists a holomorphic disk bounded by $L_1$ in a class $\beta$ of Maslov index two. By Lemma~\ref{lemma_limitofholomorphicdiscs}, there exists a class $\widehat{\beta}$ which can be realized by a holomorphic disk such that $[\widehat{\beta}] = \widehat{\Psi}_{1,*} (\beta)$ in $\overline{H}_2 (\widehat{\mcal{X}}_0, \widehat{L}_0)$.

By Lemma~\ref{lemma_3maslovcorr} and~\ref{lemma_aaa}, the Maslov index of $\widehat{\beta}$ is two and $\widehat{\beta}$ is one of the classes in~\eqref{equ_possibleclasse}. Also, Lemma~\ref{lemma_computationofogw} computes the counting invariant of $\beta_{\bullet}$ for $\bullet = 0, 1, \dots, n$ and hence the coefficients of the Laurent monomials corresponding to those classes are all one.

Observe that 
$$
\Pi_* (\widehat{\beta}_{01} + \ell (\widehat{\beta}_0 + \widehat{\beta}_1 - 2 \widehat{\beta}_{01})) = \Pi_* (\widehat{\beta}_{01}),
$$
which corresponds to the Laurent monomial of the form $z_2$ by Corollary~\ref{corollary_twodiscs}. Moreover, it follows from Gromov's compactness theorem that there are only \emph{finitely} many classes of Maslov index two represented by holomorphic disks bounded by the monotone Lagrangian torus $L_1$. Recall that we are perturbing the moduli spaces geometrically (\emph{not} abstractly), by taking a generic choice of almost complex structures on $X_1$. Thus, the counting invariants are \emph{integers}. In conclusion, the disk potential function is of the form~\eqref{equ_undeterminedpotential}.
\end{proof}

The remaining task is to determine the coefficient $\kappa$. 

\begin{proposition}\label{proposition_openGWin}
The coefficient $\kappa$ of $z_2$ in~\eqref{equ_undeterminedpotential} is two. 
\end{proposition}

\begin{proof}
First, let us compute the partial derivatives of $W_{L_1}$ in~\eqref{equ_undeterminedpotential}.
Note that ${\partial_{z_1} W_{L_1}}= 0$ yields that $z_1 = \pm 1$. The partial derivatives with respect to the other variables give rise to the following system of equations$\colon$
\begin{equation}\label{equ_equequ0}
\frac{1}{z_n} = \frac{z_n}{z_{n-1}} = \cdots = \frac{z_3}{z_2} = \frac{z_2}{z_1} + \kappa \cdot z_2 + z_1z_2,
\end{equation}
which yields 
\begin{equation}\label{equ_equequ1}
z_{n-1} = z_n^2, \, \cdots, \, z_{j} = z_n^{n-j+1}, \, \cdots, \, z_{2} = z_{n}^{n-1}
\end{equation}
and
\begin{equation}\label{equ_equequ2}
z_3 = z_2^2 \left( \frac{1}{z_1} + \kappa + z_1 \right) \Rightarrow z_n^{n-2} = (z_n^{n-1})^2 \left( \frac{1}{z_1} +\kappa + z_1 \right)
\end{equation}
By~\eqref{equ_equequ1} and~\eqref{equ_equequ2}, regardless of $\kappa$, the disk potential $W_{L_1}$ always admits a critical point $(z_1, \cdots, z_n) \in (\C^*)^n$. It means that there exists a flat connection $\nabla$ on the trivial line bundle $\scr{L}$ over $L_1$ making $(L_1, \nabla)$ a non-zero object in the monotone Fukaya category.

Then, by Corollary~\ref{cor_criticalvalues} and Theorem~\ref{theorem_corsheridan}, the possible critical values of $W_{L_1}(\mathbf{z})$ are
\begin{equation}\label{equ_equeque5}
\left\{ n \sqrt[n]{4} \cdot \zeta \mid \zeta^n = 1 \right\} \cup \{ 0\}. 
\end{equation}
Because of~\eqref{equ_equequ0}, the critical values of $W_{L_1}$ are
$$
\frac{1}{z_n} + \frac{z_n}{z_{n-1}} + \cdots + \frac{z_3}{z_2} + \frac{z_2}{z_1} + \kappa \cdot z_2 + z_1z_2 = \frac{n}{z_n}.
$$
Since $n/z_n \neq 0$, $z_n^{-1} = \sqrt[n]{4} \cdot \zeta$ and hence $z_n^{-n} = 4$. By~\eqref{equ_equequ2}, we have
\begin{equation}\label{equ_equequ4}
\frac{1}{z_1} + \kappa + z_1 = 4.
\end{equation}
If $z_1 = -1$ (resp. $z_1 = 1$), then $\kappa = 6$ (resp. $\kappa = 2$).

\textbf{Case I.}  If $\kappa = 6$, the disk potential $W_{L_1}$ \emph{does} have a critical point with $z_1 = 1$. From $z_1 = 1$,~\eqref{equ_equequ2} determines $z_n$ and then the remaining $z_\bullet$'s are determined by~\eqref{equ_equequ1}. However, both $z_1 = 1$ and~\eqref{equ_equequ4} cannot be satisfied, which in turn implies that the corresponding critical value is not in the list~\eqref{equ_equeque5}.

\textbf{Case II.} If $\kappa = 2$, the disk potential $W_{L_1}$ does \emph{not} have a critical point with $z_1 = -1$. It is because $z_3$ must vanish by~\eqref{equ_equequ2}. In this case, $W_{L_1}$ can have a critical point only when $z_1 = 1$ and thus every critical value of $W_{L_1}$ is in the list~\eqref{equ_equeque5}. 

In sum, $\kappa = 2$.
\end{proof}

\begin{proof}[Proof of Theorem~\ref{theorem_main}]
Combining Proposition~\ref{proposition_potentialfunctionm} and~\ref{proposition_openGWin}, the disk potential $W_{L}$ is explicitly computed.
\end{proof}

\section{Final remarks}

One of our motivations is to devise a tool for computing the disk potential of a monotone projective smooth variety which occurs at a generic fiber of a toric degeneration. We would like to propose a computational strategy, which we call ``from globals to locals" and ``from locals to globals". Namely, we first compute the disk potential of a simpler model and then understand the the moduli of pseudo holomorphic disks in local models contained in the simpler model. Using the local models, we compute the disk potential of another new space.  

The author expects that the strategy is applicable to a class of Fano manifolds admitting completely integrable systems compatible with toric degenerations arising from Newton--Okounkov bodies. For instance, the class includes Gelfand--Zeitlin systems of isotropic flag manifolds and bending systems of polygon spaces. 

We have computed the disk potential for $\mathrm{OG}(1, \C^{n+2}) \simeq \mcal{Q}_n$. According to~\cite{ChoKimSO}, the coadjoint orbit $\mcal{O}_\lambda \simeq \mathrm{OG}(1, \C^{n+2})$ has a Lagrangian GZ $S^n$-fiber at the origin when $\lambda$ is~\eqref{equ_givenlambda}. Thus, we obtain the disk potential for $T^*S^n$. To study GZ systems and bending systems, in addition to the local model $T^*S^n$, we also need the local model $T^* \mathrm{SO}(3)$ such that the affine base of a Lagrangian fibration on $T^* \mathrm{SO}(3)$ (or a product of bases from studied local models) locally agrees with that of the affine base of Lagrangian fibration that we want to study.

The orthogonal Grassmannian $\mathrm{OG}(2, \C^5) (\simeq \mathrm{OG}(3, \C^6))$ contains such a local model of $T^*\mathrm{SO}(3)$. As the disk potential for $\mathrm{OG}(2, \C^5)$ can be derived as in the previous sections. We briefly discuss the computation for the disk potential for future research. The author jointly with S.-C. Lau and X. Zheng will compute the disk potential of polygon spaces.

Using the same strategy in Section~\ref{sect_effectivediscclasses} and~\ref{section_computingdiscpotfuncs}, we can also compute the disk potential of the GZ system in $\mathrm{OG}(2, \C^5)$. Take $\lambda = (\lambda_1 = \lambda_ 2 )$ with $\lambda_1 > 0$. We denote by $\mcal{O}^5_\lambda$ the adjoint $\mathrm{SO}(5)$-orbit of the block diagonal matrix $D^5_\lambda = \mathrm{diag}(B(\lambda_1), B(\lambda_2), 0_{1 \times 1})$, which is isomorphic to $\mathrm{OG}(2, \C^5)$.
Moreover, it admits the GZ system on $\mcal{O}^5_\lambda$, whose image is given by 
\begin{equation}\label{equ_GZpolytopeog25}
u_{1,2} + u_{1,1} \geq 0, \, u_{1,2} - u_{1,1} \geq 0, \, u_{1,2} + u_{2,2} \geq 0, \, u_{1,2} - u_{2,2} \geq 0, \,\lambda_1 - u_{1,2} \geq 0,
\end{equation}
see \cite[Section 2.4]{ChoKimSO} for the explicit construction of the GZ system. The image is a polytope in $\R^3$ as depicted in Figure~\ref{fig_holomorphicdiscofMaslovindextwosoog25}.

There are five basic classes that can be represented by a holomorphic disk of Maslov index two emanating from the inverse image of a facet. As in Lemma~\ref{lemma_computationofogw}, the counting invariant of each basic class is one. In addition to the five basic classes, by examining the toric degenerations and its resolutions as in Proposition~\ref{proposition_potentialfunctionm}, we classify all possible boundary classes in $H^1(L_1; \Z)$ of holomorphic disks of Maslov index two, which contribute to the disk potential. The disk potential is then of the form
$$
W_{L_1} (\mathbf{z}) = \frac{1}{z_{1,2}} + \frac{z_{1,2}}{z_{2,2}} + z_{1,2} z_{2,2} + \kappa \cdot z_{1,2} + \frac{z_{1,2}}{z_{1,1}} + z_{1,2}z_{1,1}.
$$
As $\mathrm{OG}(3,\C^6) \simeq \mathrm{OG}(2, \C^5) \simeq \CP^3$, the list of critical values is $4, -4, 4  \sqrt{-1},$ and $- 4  \sqrt{-1}$. The only possible $\kappa$ such that the set of critical values of $W_{L_1}$ is contained in $\{\pm4, \pm 4 \sqrt{-1}\}$ is $0$. Therefore, we derive the disk potential of the GZ system for $\mathrm{OG}(2, \C^5)$.

\begin{theorem}
The disk potential function of the monotone Gelfand--Zeitlin fiber for $\mathrm{OG}(2, \C^5)$ is 
$$
W_{L_1} (\mathbf{z}) = \frac{1}{z_{1,2}} + \frac{z_{1,2}}{z_{2,2}} + z_{1,2} z_{2,2} + \frac{z_{1,2}}{z_{1,1}} + z_{1,2}z_{1,1}.
$$
\end{theorem}

\begin{figure}[h]
	\scalebox{0.85}{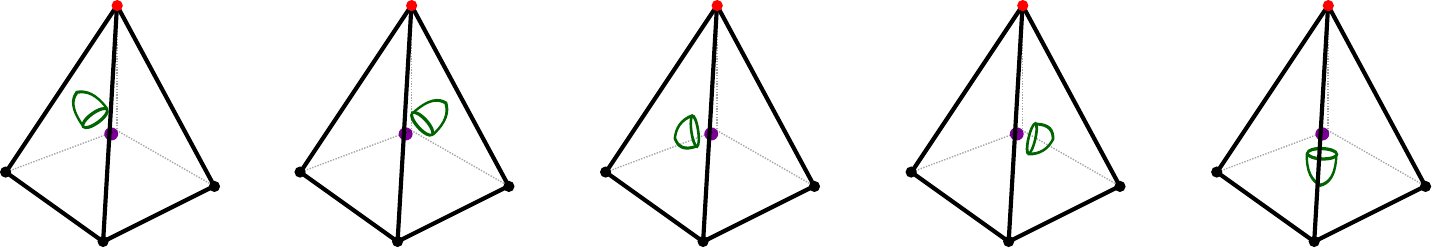}
	\caption{\label{fig_holomorphicdiscofMaslovindextwosoog25} Holomorphic disks of Maslov index two bounded by $L_1$ in $\mathrm{OG}(2, \C^5)$}	
\end{figure}

The GZ fiber at the tip of the pyramid ($u_{1,1} = u_{1,2} = u_{2,1} = 0$) in Figure~\ref{fig_holomorphicdiscofMaslovindextwosoog25} is diffeomorphic to $\mathrm{SO}(3)$. Moreover, the complement of the divisor over $\lambda_1 - u_{2,1} = 0$ is $T^*\mathrm{SO}(3)$. From the global model, we can understand the moduli spaces of holomorphic disks of Maslov index two in $T^*\mathrm{SO}(3)$ bounded by a Lagrangian torus which degenerates into a toric fiber in the chart around the red tip point of the toric variety associated to the polytope given by~\eqref{equ_GZpolytopeog25}. We then obtain another local model $T^* \mathrm{SO}(3)$, which is sitting inside of $\mathrm{OG}(2, \C^5)$.

\providecommand{\bysame}{\leavevmode\hbox to3em{\hrulefill}\thinspace}
\providecommand{\MR}{\relax\ifhmode\unskip\space\fi MR }
\providecommand{\MRhref}[2]{%
  \href{http://www.ams.org/mathscinet-getitem?mr=#1}{#2}
}
\providecommand{\href}[2]{#2}


\begin{thebibliography}{BCFKvS00}

\bibitem[Aur07]{AurouxTdual}
Denis Auroux, \emph{Mirror symmetry and {$T$}-duality in the complement of an
  anticanonical divisor}, J. G\"{o}kova Geom. Topol. GGT \textbf{1} (2007),
  51--91. \MR{2386535}

\bibitem[Aur09]{AurouxSpecial}
\bysame, \emph{Special {L}agrangian fibrations, wall-crossing, and mirror
  symmetry}, Surveys in differential geometry. {V}ol. {XIII}. {G}eometry,
  analysis, and algebraic geometry: forty years of the {J}ournal of
  {D}ifferential {G}eometry, Surv. Differ. Geom., vol.~13, Int. Press,
  Somerville, MA, 2009, pp.~1--47. \MR{2537081}

\bibitem[BC12]{BiranCornea}
Paul Biran and Octav Cornea, \emph{Lagrangian topology and enumerative
  geometry}, Geom. Topol. \textbf{16} (2012), no.~2, 963--1052. \MR{2928987}

\bibitem[BCFKvS00]{BCFKvS}
Victor~V. Batyrev, Ionu\c{t} Ciocan-Fontanine, Bumsig Kim, and Duco van
  Straten, \emph{Mirror symmetry and toric degenerations of partial flag
  manifolds}, Acta Math. \textbf{184} (2000), no.~1, 1--39. \MR{1756568}

\bibitem[BZ01]{BerensteinZelevinsky}
Arkady Berenstein and Andrei Zelevinsky, \emph{Tensor product multiplicities,
  canonical bases and totally positive varieties}, Invent. Math. \textbf{143}
  (2001), no.~1, 77--128. \MR{1802793}

\bibitem[Cal02]{Caldero}
Philippe Caldero, \emph{Toric degenerations of {S}chubert varieties},
  Transform. Groups \textbf{7} (2002), no.~1, 51--60. \MR{1888475}

\bibitem[Cho04]{Cho}
Cheol-Hyun Cho, \emph{Holomorphic discs, spin structures, and {F}loer
  cohomology of the {C}lifford torus}, Int. Math. Res. Not. (2004), no.~35,
  1803--1843. \MR{2057871}

\bibitem[Cho08]{Chononunitrary}
\bysame, \emph{Non-displaceable {L}agrangian submanifolds and {F}loer
  cohomology with non-unitary line bundle}, J. Geom. Phys. \textbf{58} (2008),
  no.~11, 1465--1476. \MR{2463805}

  \bibitem[CK21]{ChoKimMONO}
Yunhyung Cho and Yoosik Kim, \emph{Monotone {L}agrangians in flag varieties},
  Int. Math. Res. Not. IMRN (2021), no.~18, 13892--13945. \MR{4320800}

\bibitem[CK20]{ChoKimSO}
\bysame, \emph{Lagrangian fibers of {G}elfand--{C}etlin systems of
  ${SO}(n)$-type}, Transform. Groups \textbf{25} (2020), no.~4, 1063--1102.
  \MR{4166681}

\bibitem[CKLP23]{ChoKimLeePark}
Yunhyung Cho, Yoosik Kim, Eunjeong Lee, and Kyeong-Dong Park, \emph{Small toric
  resolutions of toric varieties of string polytopes with small indices}, Commun. Contemp. 
  Math. \textbf{25} (2023), no.~1, Paper No. 2150112, 56 pp. \MR{4523149}

\bibitem[CKO20]{ChoKimOh1}
Yunhyung Cho, Yoosik Kim, and Yong-Geun Oh, \emph{Lagrangian fibers of
  {G}elfand-{C}etlin systems}, Adv. Math. \textbf{372} (2020), 107304, 57 pp.
  \MR{4126721}

\bibitem[CLL12]{ChanLauLeung}
Kwokwai Chan, Siu-Cheong Lau, and Naichung~Conan Leung, \emph{S{YZ} mirror
  symmetry for toric {C}alabi-{Y}au manifolds}, J. Differential Geom.
  \textbf{90} (2012), no.~2, 177--250. \MR{2899874}

\bibitem[CLLT17]{ChanLauLeungTseng}
Kwokwai Chan, Siu-Cheong Lau, Naichung~Conan Leung, and Hsian-Hua Tseng,
  \emph{Open {G}romov-{W}itten invariants, mirror maps, and {S}eidel
  representations for toric manifolds}, Duke Math. J. \textbf{166} (2017),
  no.~8, 1405--1462. \MR{3659939}

\bibitem[CO06]{ChoOh}
Cheol-Hyun Cho and Yong-Geun Oh, \emph{Floer cohomology and disc instantons of
  {L}agrangian torus fibers in {F}ano toric manifolds}, Asian J. Math.
  \textbf{10} (2006), no.~4, 773--814. \MR{2282365}

\bibitem[CP14]{ChoPoddar}
Cheol-Hyun Cho and Mainak Poddar, \emph{Holomorphic orbi-discs and {L}agrangian
  {F}loer cohomology of symplectic toric orbifolds}, J. Differential Geom.
  \textbf{98} (2014), no.~1, 21--116. \MR{3263515}

\bibitem[CW22]{CharestWoodward}
Fran{\c c}ois Charest and Chris~T. Woodward, \emph{Floer cohomology and flips}, 
Mem. Amer. Math. Soc. \textbf{279} (2022), no.~1372, v+166 pp. \MR{4464438}

\bibitem[EP97]{EliashbergPolterovich}
Yakov Eliashberg and Leonid Polterovich, \emph{The problem of {L}agrangian
  knots in four-manifolds}, Geometric topology ({A}thens, {GA}, 1993), AMS/IP
  Stud. Adv. Math., vol.~2, Amer. Math. Soc., Providence, RI, 1997,
  pp.~313--327. \MR{1470735}

\bibitem[FOOO09a]{FOOOMorse}
Kenji Fukaya, Yong-Geun Oh, Hiroshi Ohta, and Kaoru Ono, \emph{Canonical models
  of filtered {$A_\infty$}-algebras and {M}orse complexes}, New perspectives
  and challenges in symplectic field theory, CRM Proc. Lecture Notes, vol.~49,
  Amer. Math. Soc., Providence, RI, 2009, pp.~201--227. \MR{2555938}

\bibitem[FOOO09b]{FOOObook1}
\bysame, \emph{Lagrangian intersection {F}loer theory: anomaly and obstruction.
  {P}art {I}}, AMS/IP Studies in Advanced Mathematics, vol.~46, American
  Mathematical Society, Providence, RI; International Press, Somerville, MA,
  2009. \MR{2553465}

\bibitem[FOOO10]{FOOOToric1}
\bysame, \emph{Lagrangian {F}loer theory on compact toric manifolds. {I}}, Duke
  Math. J. \textbf{151} (2010), no.~1, 23--174. \MR{2573826}

\bibitem[FOOO11]{FOOOToric2}
\bysame, \emph{Lagrangian {F}loer theory on compact toric manifolds {II}: bulk
  deformations}, Selecta Math. (N.S.) \textbf{17} (2011), no.~3, 609--711.
  \MR{2827178}

\bibitem[FOOO12]{FOOOS2S2}
\bysame, \emph{Toric degeneration and nondisplaceable {L}agrangian tori in
  {$S^2\times S^2$}}, Int. Math. Res. Not. IMRN (2012), no.~13, 2942--2993.
  \MR{2946229}

\bibitem[Fuk10]{FukayaCyclic}
Kenji Fukaya, \emph{Cyclic symmetry and adic convergence in {L}agrangian
  {F}loer theory}, Kyoto J. Math. \textbf{50} (2010), no.~3, 521--590.
  \MR{2723862}

\bibitem[GL96]{GonciuleaLakshmibai}
N.~Gonciulea and V.~Lakshmibai, \emph{Degenerations of flag and {S}chubert
  varieties to toric varieties}, Transform. Groups \textbf{1} (1996), no.~3,
  215--248. \MR{1417711}

\bibitem[GS83a]{GuilleminSternbergGC}
Victor Guillemin and Shlomo Sternberg, \emph{The {G}el'fand-{C}etlin system and
  quantization of the complex flag manifolds}, J. Funct. Anal. \textbf{52}
  (1983), no.~1, 106--128. \MR{705993}

\bibitem[GS83b]{GuilleminSternbergCCI}
\bysame, \emph{On collective complete integrability according to the method of
  {T}himm}, Ergodic Theory Dynam. Systems \textbf{3} (1983), no.~2, 219--230.
  \MR{742224}

\bibitem[HK97]{HausmannKnutson}
Jean-Claude Hausmann and Allen Knutson, \emph{Polygon spaces and
  {G}rassmannians}, Enseign. Math. (2) \textbf{43} (1997), no.~1-2, 173--198.
  \MR{1460127}

\bibitem[HK15]{HaradaKaveh}
Megumi Harada and Kiumars Kaveh, \emph{Integrable systems, toric degenerations
  and {O}kounkov bodies}, Invent. Math. \textbf{202} (2015), no.~3, 927--985.
  \MR{3425384}

\bibitem[HKL19]{HongKimLau}
Hansol Hong, Yoosik Kim, and Siu-Cheong Lau, \emph{Immersed two-spheres and
  {SYZ} with {A}pplication to {G}rassmannians}, preprint (2019),
  \href{https://arxiv.org/abs/1805.11738}{arXiv:1805.11738}.

\bibitem[Kar99]{KarshonPHF}
Yael Karshon, \emph{Periodic {H}amiltonian flows on four-dimensional
  manifolds}, Mem. Amer. Math. Soc. \textbf{141} (1999), no.~672, viii+71.
  \MR{1612833}

\bibitem[KM96]{KapovichMillson}
Michael Kapovich and John~J. Millson, \emph{The symplectic geometry of polygons
  in {E}uclidean space}, J. Differential Geom. \textbf{44} (1996), no.~3,
  479--513. \MR{1431002}

\bibitem[Lan20]{Lanesympcont}
Jeremy Lane, \emph{The geometric structure of symplectic contraction}, Int.
  Math. Res. Not. IMRN (2020), no.~12, 3521--3539. \MR{4120303}

\bibitem[Lit98]{Littelmann}
P.~Littelmann, \emph{Cones, crystals, and patterns}, Transform. Groups
  \textbf{3} (1998), no.~2, 145--179. \MR{1628449}

\bibitem[NNU10a]{NishinouNoharaUeda2}
Takeo Nishinou, Yuichi Nohara, and Kazushi Ueda, \emph{Potential functions via
  toric degenerations}, preprint (2010),
  \href{https://arxiv.org/abs/0812.0066}{arXiv:0812.0066}.

\bibitem[NNU10b]{NishinouNoharaUeda}
\bysame, \emph{Toric degenerations of {G}elfand-{C}etlin systems and potential
  functions}, Adv. Math. \textbf{224} (2010), no.~2, 648--706. \MR{2609019}

\bibitem[NU14]{NoharaUedaGrPoly}
Yuichi Nohara and Kazushi Ueda, \emph{Toric degenerations of integrable systems
  on {G}rassmannians and polygon spaces}, Nagoya Math. J. \textbf{214} (2014),
  125--168. \MR{3211821}

\bibitem[Oh93a]{OhMonotone}
Yong-Geun Oh, \emph{Floer cohomology of {L}agrangian intersections and
  pseudo-holomorphic disks. {I}}, Comm. Pure Appl. Math. \textbf{46} (1993),
  no.~7, 949--993. \MR{1223659}

\bibitem[Oh93b]{OhMonotone2}
\bysame, \emph{Floer cohomology of {L}agrangian intersections and
  pseudo-holomorphic disks. {II}. {$({\bf C}{\rm P}^n,{\bf R}{\rm P}^n)$}},
  Comm. Pure Appl. Math. \textbf{46} (1993), no.~7, 995--1012. \MR{1223660}

\bibitem[Oh96]{OhPearl}
\bysame, \emph{Relative {F}loer and quantum cohomology and the symplectic
  topology of {L}agrangian submanifolds}, Contact and symplectic geometry
  ({C}ambridge, 1994), Publ. Newton Inst., vol.~8, Cambridge Univ. Press,
  Cambridge, 1996, pp.~201--267. \MR{1432465}

\bibitem[OU16]{OakleyUsher}
Joel Oakley and Michael Usher, \emph{On certain {L}agrangian submanifolds of
  {$S^2\times S^2$} and {$\Bbb{C}{\rm P}^n$}}, Algebr. Geom. Topol. \textbf{16}
  (2016), no.~1, 149--209. \MR{3470699}
  
\bibitem[Pab12]{Pabiniak}  
Milena Pabiniak, \emph{Hamiltonian torus actions in equivariant cohomology and
 symplectic topology}, Thesis (Ph.D.)--Cornell University, (2012), 125 pp. \MR{3103586}

\bibitem[PRW16]{PechRietschWilliams}
C.~Pech, K.~Rietsch, and L.~Williams, \emph{On {L}andau-{G}inzburg models for
  quadrics and flat sections of {D}ubrovin connections}, Adv. Math.
  \textbf{300} (2016), 275--319. \MR{3534834}

\bibitem[Prz13]{Przhiya}
V.~V. Przhiyalkovski\u{\i}, \emph{Weak {L}andau-{G}inzburg models of smooth
  {F}ano threefolds}, Izv. Ross. Akad. Nauk Ser. Mat. \textbf{77} (2013),
  no.~4, 135--160. \MR{3135701}

\bibitem[Rie08]{Rietsch}
Konstanze Rietsch, \emph{A mirror symmetric construction of
  {$qH^\ast_T(G/P)_{(q)}$}}, Adv. Math. \textbf{217} (2008), no.~6, 2401--2442.
  \MR{2397456}

\bibitem[Rua02]{WERuan2}
Wei-Dong Ruan, \emph{Lagrangian torus fibration of quintic {C}alabi-{Y}au
  hypersurfaces. {II}. {T}echnical results on gradient flow construction}, J.
  Symplectic Geom. \textbf{1} (2002), no.~3, 435--521. \MR{1959057}

\bibitem[She16]{SheridanFano}
Nick Sheridan, \emph{On the {F}ukaya category of a {F}ano hypersurface in
  projective space}, Publ. Math. Inst. Hautes \'{E}tudes Sci. \textbf{124}
  (2016), 165--317. \MR{3578916}

\bibitem[Thi81]{Thimm}
A.~Thimm, \emph{Integrable geodesic flows on homogeneous spaces}, Ergodic
  Theory Dynam. Systems \textbf{1} (1981), no.~4, 495--517 (1982). \MR{662740}

\bibitem[Woo11]{WoodwardToric}
Christopher~T. Woodward, \emph{Gauged {F}loer theory of toric moment fibers},
  Geom. Funct. Anal. \textbf{21} (2011), no.~3, 680--749. \MR{2810861}

\end{thebibliography}
\end{document}